\documentclass[10pt,letterpaper]{amsart}
\usepackage{amsfonts, amsmath, amssymb, amscd, amsthm, graphicx,enumerate}
\usepackage{color}
\usepackage{hyperref}

\makeatletter
\def\blfootnote{\xdef\@thefnmark{}\@footnotetext}
\makeatother

\newtheorem{thm}{Theorem}[section]
\newtheorem{mainthm}{Theorem}

\newtheorem{cor}[thm]{Corollary}
\newtheorem{lem}[thm]{Lemma}
\newtheorem{prop}[thm]{Proposition}

\theoremstyle{definition}
\newtheorem{defn}[thm]{Definition}
\theoremstyle{remark}
\newtheorem{rem}[thm]{Remark}
\newtheorem{ex}[thm]{Example}

\newtheorem{conv}[thm]{Convention}
\newtheorem{nota}[thm]{Notation}

\renewcommand{\phi}{\varphi}
\renewcommand{\epsilon}{\varepsilon}
\newcommand{\N}{\mathbb N}
\newcommand{\R}{\mathbb R}
\newcommand{\Z}{\mathbb Z}
\renewcommand{\Pr}{\mathbb P}
\newcommand{\Aut}{\operatorname{Aut}}
\newcommand{\Out}{\operatorname{Out}}
\newcommand{\Stab}{\operatorname{Stab}}
\newcommand{\Fix}{\operatorname{Fix}}
\newcommand{\rank}{\operatorname{rank}}
\newcommand{\supp}{\operatorname{supp}}
\newcommand{\Curr}{\operatorname{Curr}}
\newcommand{\PCN}{\Pr\Curr(F_N)}
\newcommand{\CN}{\Curr(F_N)}
\newcommand{\cvn}{\mathrm{cv}_N}
\newcommand{\cvnbar}{\overline{\mathrm{cv}}_N}
\newcommand{\sm}{\langle \supp(\mu)\rangle_+}
\newcommand{\dd}{\partial^2 F_N}

\begin{document}
\raggedbottom

\title{Generic-case complexity of Whitehead's algorithm, revisited}

\author{Ilya Kapovich}

\address{Department of Mathematics and Statistics, Hunter College of CUNY\newline
  \indent 695 Park Ave, New York, NY 10065
  \newline \indent  {\url{http://math.hunter.cuny.edu/ilyakapo/}}, }
  \email{\tt ik535@hunter.cuny.edu}

\thanks{The author was supported by the individual NSF
  grant DMS-1905641}


\date{}
\hypersetup{
  pdftitle={Generic-case complexity of Whitehead's algorithm, revisited},
  pdfauthor={Ilya Kapovich},
  pdfsubject={Whitehead's algorithm, geodesic currents, and random walks},
  pdfkeywords={free groups, Whitehead's algorithm, random walks, geodesic currents}
}

\begin{abstract}
The results of \cite{KSS06} show that, with respect to the simple non-backtracking random walk on the free group $F_N=F(a_1,\dots,a_N)$, Whitehead's algorithm has strongly linear-time generic-case complexity and that generic elements of $F_N$ are strictly minimal in their $\Out(F_N)$-orbits.
Here we generalize these results, with appropriate modifications, to a much wider class of random processes generating elements of $F_N$. We introduce the notion of an \emph{$M$-minimal} conjugacy class $[w]$ in $F_N$, where $M\ge 1$. For $[w]$ to be $M$-minimal means that any chain of $k$ Whitehead moves, starting with $[w]$ and producing a sequence of distinct conjugacy classes with monotonically non-increasing lengths, satisfies $k\le M$. We prove that if a conjugacy class $[w]$ is sufficiently close to a filling projective geodesic current $[\nu]\in\PCN$, then, after applying one of finitely many reducing automorphisms depending only on $\nu$, one obtains an $M$-minimal conjugacy class, for a uniform constant $M=M(\nu)$. Consequently, the corresponding finite speed-up of Whitehead's algorithm works in quadratic time when the other input is arbitrary, and in linear time when both inputs are projectively close to $[\nu]$. We then prove that a wide class of random processes produces random conjugacy classes $[w_n]$ that converge projectively to a filling current. For such $[w_n]$, if $u\in F_N$ is arbitrary, the finite speed-up of Whitehead's algorithm on $([w_n],[u])$ works in time $O(\max\{|w_n|,|u|^2\})$; if $[w_m']$ is also generated randomly, it works in time $O(\max\{|w_n|,|w_m'|\})$.
\end{abstract}

\maketitle

\tableofcontents



\section{Introduction}

\subsection{Whitehead's algorithm and generic complexity}

Let $F_N=F(A)$ be a free group of finite rank $N\ge2$, with a fixed free basis $A=\{a_1,\dots,a_N\}$. The \emph{automorphism problem} for $F_N$ asks, given two freely reduced words $w,w'\in F_N$, whether $w'=\Phi(w)$ for some $\Phi\in\Aut(F_N)$. For $g\in F_N$, let $|g|_A$ and $\|g\|_A$ denote its freely reduced and cyclically reduced lengths, respectively, and let $[g]$ denote its conjugacy class. Put
\[
\mathcal C_N:=\{[g]\mid g\in F_N\}.
\]
Then
\[
\Aut(F_N)w=\Aut(F_N)w'
\quad\Longleftrightarrow\quad
\Out(F_N)[w]=\Out(F_N)[w'].
\]
Thus we formulate the automorphism problem in terms of the action of $\Out(F_N)$ on $\mathcal C_N$.

Whitehead~\cite{W36} solved this problem using a finite generating set of \emph{Whitehead automorphisms} for $\Aut(F_N)$. Their images in $\Out(F_N)$ are called \emph{Whitehead moves}; the set of nontrivial such outer automorphisms is denoted by $\mathcal W_N$, see Definition~\ref{defn:moves}. A conjugacy class $[u]\in\mathcal C_N$ is \emph{$\Out(F_N)$-minimal} if
\[
\|u\|_A\le \|\phi(u)\|_A
\qquad\text{for every }\phi\in\Out(F_N),
\]
and it is \emph{Whitehead-minimal} if the same inequality holds for every $\tau\in\mathcal W_N$. Whitehead's "peak reduction" theorem  implies that these two conditions are equivalent. Peak reduction also divides Whitehead's algorithm into two stages. First, one repeatedly applies a Whitehead move that strictly decreases cyclically reduced length until an $\Out(F_N)$-minimal conjugacy class is reached. This minimization stage has a general quadratic-time bound. One then explores the component of minimal representatives connected by length-preserving Whitehead moves. This stabilization component may have exponentially many vertices, and the best general upper bound for the complete algorithm is exponential. The precise statements are recalled in Section~\ref{s:wh}. The worst-case complexity remains open in rank $N\ge3$; see, among others, \cite{Ci07,Kh04,Lee1,Lee,MH06,MS03,RVW}. In rank $N=2$ a quadratic time bound is known~\cite{MS03,Kh04}.

Kapovich--Schupp--Shpilrain~\cite{KSS06} initiated the generic-case complexity study of Whitehead's algorithm. They called $[w]\in\mathcal C_N$ \emph{strictly minimal} if
\[
\|w\|_A<\|\tau(w)\|_A
\]
for every non-inner Whitehead move $\tau$ of the second kind. They proved~\cite{KSS06} that a uniformly random freely reduced or cyclically reduced word in $F_N$ is strictly minimal with exponentially high probability. Consequently, both stages of Whitehead's algorithm are generically linear when both inputs are random, while for one random input $w$ and one arbitrary input $u$ the running time is
\[
O\bigl(\max\{|w|_A,|u|_A^2\}\bigr).
\]

Strict minimality is too rigid for many other natural random processes. For example, let $F_2=F(a,b)$ and let $w_n$ be a positive Bernoulli word of length $n$ with probabilities $p(a)=1/10$ and $p(b)=9/10$. For the Whitehead automorphism
\[
\tau(a)=ab^{-1},\qquad \tau(b)=b,
\]
let $A_n$ be the number of occurrences of $a$ and let $C_n$ be the number of cyclic occurrences of $ab$ in $w_n$. Then
\[
\|\tau(w_n)\|_A=n+A_n-2C_n,
\qquad
\frac{\|\tau(w_n)\|_A}{\|w_n\|_A}
\longrightarrow 1+\frac1{10}-2\frac9{100}=\frac{23}{25}<1
\]
almost surely. Thus $[w_n]$ is not even Whitehead-minimal with probability tending to $1$. This example motivates a weaker notion that still controls the complexity of Whitehead's algorithm.

\subsection{\texorpdfstring{$M$-minimality}{M-minimality}}

The following is a key technical notion introduced in this paper. It generalizes strict minimality by replacing the requirement that every relevant Whitehead move increase length with a uniform bound on all simple non-increasing chains of Whitehead moves.

Let $M\ge1$. A conjugacy class $[u]\in\mathcal C_N$ is called \emph{$M$-minimal} (see Definition~\ref{d:M} below) if the following condition holds. Whenever $\tau_1,\dots,\tau_k\in\mathcal W_N$ and
\[
[u_i]:=\tau_i\cdots\tau_1([u])
\qquad(1\le i\le k)
\]
satisfy
\[
\|u\|_A\ge \|u_1\|_A\ge\cdots\ge\|u_k\|_A
\]
and the $k+1$ conjugacy classes
\[
[u],[u_1],\dots,[u_k]
\]
are pairwise distinct, then $k\le M$.  In particular, every strictly length-decreasing Whitehead-minimization chain starting at $[u]$ has length at most $M$, and the set of $\Out(F_N)$-minimal representatives in the orbit of $[u]$ has uniformly bounded cardinality in terms of $N, M$.

We prove:
\begin{mainthm}\label{thm:introA}
Fix $N\ge2$ and $M\ge1$. There exist constants $C=C(N,M)\ge1$ and $K=K(N,M)\ge1$ such that the following hold.
\begin{enumerate}
\item Given $1\ne u\in F_N$, whether $[u]$ is $M$-minimal can be decided in time at most $K|u|_A$.
\item If $1\ne u\in F_N$ and $[u]$ is $M$-minimal, then Whitehead minimization on $u$ terminates in time at most $K|u|_A$, and
\[
\#\mathcal M([u])\le C,
\qquad
\rank\,\Stab_{\Out(F_N)}([u])\le C.
\]
\item If $1\ne u,v\in F_N$ and $[u]$ and $[v]$ are both $M$-minimal, then Whitehead's algorithm decides whether
\[
\Out(F_N)[u]=\Out(F_N)[v]
\]
in time at most
\[
K\max\{|u|_A,|v|_A\}.
\]
\item If $1\ne u\in F_N$, $[u]$ is $M$-minimal, and $1\ne v\in F_N$ is arbitrary, then the same problem can be decided in time at most
\[
K\max\{|u|_A,|v|_A^2\}.
\]
\end{enumerate}
\end{mainthm}

The cases involving the trivial conjugacy class are immediate: $[1]$ is $M$-minimal, $\mathcal M([1])=\{[1]\}$, and $\Stab_{\Out(F_N)}([1])=\Out(F_N)$, which is generated by the finite set of Whitehead moves. Inputs equal to the identity can therefore be recognized and handled separately in constant time.

Theorem~\ref{thm:introA} combines Lemma~\ref{lem:M}, Theorem~\ref{t:WHM}, and Propositions~\ref{prop:st} and~\ref{p:algM}. Every strictly minimal conjugacy class is $M$-minimal for $M=2^N N!$, so the new notion contains the principal generic class from \cite{KSS06}.

The proofs use two quantitative refinements, called $(M,\lambda,\epsilon)$-minimality and $(M,\lambda,\epsilon,\mathcal W_N)$-minimality. At a schematic level, they require a finite set $S\subseteq\Out(F_N)[u]$ with $\#S\le M$ such that
\[
1-\epsilon\le \frac{\|u'\|_A}{\|u\|_A}\le1+\epsilon
\qquad([u],[u']\in S),
\]
while leaving $S$ produces a definite multiplicative length increase:
\[
\frac{\|\phi(u)\|_A}{\|u\|_A}\ge\lambda
\qquad\text{whenever }[u]\in S\text{ and }\phi([u])\notin S.
\]
For $(M,\lambda,\epsilon,\mathcal W_N)$-minimality, the latter condition is required only for $\phi\in\mathcal W_N$. The precise definitions, including the orbit conditions on $S$, and their relationship to $M$-minimality are given in Section~\ref{s:MLE}.

\subsection{Filling currents and preferred representatives}

\emph{Geodesic currents} on free groups provide a crucial technical tool in this paper (Sections 4 and 5 recall the required background).  Put
\[
\partial^2F_N:=(\partial F_N\times\partial F_N)\setminus\{(x,x)\mid x\in\partial F_N\},
\]
and let $\varpi:\partial^2F_N\to\partial^2F_N$ be the flip map $\varpi(x,y)=(y,x)$. A \emph{geodesic current} on $F_N$ is a positive, locally finite Borel measure $\nu$ on $\partial^2F_N$ that is invariant under both the diagonal action of $F_N$ and the flip $\varpi$. The space of geodesic currents is denoted by $\CN$, and its projectivization is
\[
\PCN=(\CN\setminus\{0\})/\R_{>0}.
\]
For $1\ne w\in F_N$, the counting current $\eta_w\in\CN$ depends only on the conjugacy class $[w]$.

Kapovich--Lustig~\cite{KL09,KL10} constructed a continuous geometric intersection form
\[
\langle\,\cdot\,,\,\cdot\,\rangle:\cvnbar\times\CN\longrightarrow\R_{\ge0},
\]
where $\cvnbar$ is the closure of unprojectivized Culler--Vogtmann Outer space. It satisfies
\[
\langle T,\eta_w\rangle=\|w\|_T
\qquad(T\in\cvnbar,\,1\ne w\in F_N).
\]
If $T_A$ is the Cayley tree of $F_N$ corresponding to the basis $A$, put
\[
\|\nu\|_A:=\langle T_A,\nu\rangle;
\]
then $\|\eta_w\|_A=\|w\|_A$. A nonzero current $\nu\in\CN$ is called \emph{filling} if
\[
\langle T,\nu\rangle>0
\qquad\text{for every }T\in\cvnbar.
\]

Our main general result relating geodesic currents and Whitehead's algorithm is the following:
\begin{mainthm}\label{thm:introB}
Let $F_N=F(A)$ and let $0\ne\nu\in\CN$ be a filling current. There exist an integer $M\ge1$, a finite set
\[
\mathfrak W\subseteq\Out(F_N),\qquad \#\mathfrak W=M,
\]
and a number $\lambda$ with $1<\lambda<2$ such that the following holds. For every $0<\epsilon<\lambda-1$, there is a neighborhood $U$ of $[\nu]$ in $\PCN$ such that, whenever $1\ne w\in F_N$ and $[\eta_w]\in U$, the set
\[
S:=\mathfrak W[w]
\]
satisfies:
\begin{enumerate}
\item $\#S\le M$, and every element of $S$ is $M$-minimal;
\item for every $[u],[u']\in S$,
\[
1-\epsilon\le \frac{\|u'\|_A}{\|u\|_A}\le1+\epsilon;
\]
\item if $[u]\in S$ and $\phi([u])\notin S$, where $\phi\in\Out(F_N)$, then
\[
\frac{\|\phi(u)\|_A}{\|u\|_A}\ge\lambda;
\]
\item $\mathcal M([w])\subseteq S$, and $\Stab_{\Out(F_N)}([w])$ is finite.
\end{enumerate}
\end{mainthm}

Theorem~\ref{thm:introB} is proved in Theorem~\ref{t:key1} and Corollary~\ref{cor:key1}. The number $\lambda$ may be chosen below $2$, since the argument only requires it to lie strictly below the positive multiplicative gap. Thus $0<\epsilon<\lambda-1$ automatically implies $\epsilon<1$, as required in Corollary~\ref{cor:key1}. The finite set $\mathfrak W$ consists of the automorphisms minimizing $\|\phi\nu\|_A$ over $\phi\in\Out(F_N)$. Properness and discreteness of the $\Out(F_N)$-orbit of a filling current yield the gap, and continuity transfers it to counting currents whose projective classes are close to $[\nu]$.

\subsection{Adapted random processes and generic complexity}

To obtain generic-case complexity results we introduce the following notion relating random processes generating word-inputs in $F_N$ and geodesic currents.

Let $(\Omega,\Pr)$ be the probability space underlying an $F_N$-valued random process $\mathcal W=W_1,W_2,\dots$. For $0\ne\nu\in\CN$, the process $\mathcal W$ is called \emph{$\nu$-adapted} if, for $\Pr$-almost every $\omega\in\Omega$, one has $W_n(\omega)\ne1$ for all sufficiently large $n$ and
\[
[\eta_{W_n(\omega)}]\longrightarrow[\nu]
\qquad\text{in }\PCN.
\]
The process $\mathcal W$ is called \emph{tame} if there exists $C>0$ such that
\[
|W_n(\omega)|_A\le Cn
\qquad\text{for every }n\ge1\text{ and every }\omega\in\Omega.
\]
These notions are formally introduced in Definition~\ref{d:adapt}. For a finite set $\mathfrak W\subseteq\Out(F_N)$, the $\mathfrak W$-speed-up of Whitehead minimization applies the minimization procedure in parallel to $[W_n]$ and to the finitely many conjugacy classes $\phi([W_n])$, $\phi\in\mathfrak W$; the precise algorithm is defined in Section~\ref{s:wh}.

Our main general result regarding random processes and the generic-case complexity of Whitehead's algorithm is:
\begin{mainthm}\label{thm:introC}
Let $\mathcal W=W_1,W_2,\dots$ be adapted to a filling current $\nu$. Then there exist $M\ge1$, a number $\lambda$ with $1<\lambda<2$, and a finite set $\mathfrak W\subseteq\Out(F_N)$, with $\#\mathfrak W\le M$, such that for every $0<\epsilon<\lambda-1$ the following hold.
\begin{enumerate}
\item For $\Pr$-almost every trajectory, all sufficiently large $n$ have the property that $S_n:=\mathfrak W[W_n]$ satisfies conclusions (1)--(4) of Theorem~\ref{thm:introB}.
\item The probability that $S_n$ satisfies those conclusions tends to $1$ as $n\to\infty$.
\end{enumerate}
If, in addition, $\mathcal W$ is tame, there exists $K\ge1$ such that:
\begin{enumerate}
\item[(3)] for $\Pr$-almost every trajectory and all sufficiently large $n$, the $\mathfrak W$-speed-up of Whitehead minimization on $W_n$ terminates in time at most $Kn$;
\item[(4)] for $\Pr$-almost every trajectory, all sufficiently large $n$, and every $u\in F_N$, the $\mathfrak W$-speed-up decides whether
\[
\Out(F_N)[W_n]=\Out(F_N)[u]
\]
in time at most
\[
K\max\{n,|u|_A^2\};
\]
\item[(5)] for $\Pr\times\Pr$-almost every pair of independent trajectories and all sufficiently large $n,m$, it decides whether
\[
\Out(F_N)[W_n]=\Out(F_N)[W_m']
\]
in time at most
\[
K\max\{n,m\}.
\]

\item[(6)] The probability that the conclusion of (3) holds at time $n$ tends to $1$ as $n\to\infty$, and the same is true for the uniform-in-$u$ conclusion of (4).
\item[(7)] If $\mathcal W'=W_1',W_2',\dots$ is an independent copy and $n_i,m_i\ge1$ satisfy $\min\{n_i,m_i\}\to\infty$, then the probability that the conclusion of (5) holds for $(W_{n_i},W_{m_i}')$ tends to $1$.
\end{enumerate}
\end{mainthm}

The qualitative and quantitative parts of Theorem~\ref{thm:introC} are Theorems~\ref{t:A} and~\ref{t:A'}, respectively. We state the results directly in terms of random processes and probabilities, rather than choosing a single formalism for generic-case complexity. In the broader sense developed after \cite{KMSS03}, these are generic-case complexity statements.

\subsection{Group and graph random walks}

Theorem~\ref{thm:introC} becomes useful once one has natural random processes adapted to filling currents. The paper treats two broad families of such processes.

\begin{mainthm}\label{thm:introD}
The following hold.
\begin{enumerate}
\item Let $\mu:F_N\to[0,1]$ be a finitely supported probability measure such that
\[
\langle\supp(\mu)\rangle_+=F_N.
\]
Then the associated group random walk is tame and adapted to a filling current.
\item Let $\Gamma$ be a marked finite connected graph with all vertices of degree at least $3$ and $\pi_1(\Gamma)\cong F_N$, let $\mathcal X$ be an irreducible $\Gamma$-based finite-state Markov chain, let $\mu$ be any initial distribution, and let $\mathcal B$ be a closing-path system. Then both the $\mathcal B$-closing process and the modified $\mathcal B$-closing process are tame and adapted to the characteristic current $\nu_{\mathcal X}$.
\item If $\nu_{\mathcal X}$ is filling, in particular under any of the explicit hypotheses of Proposition~\ref{p:XF}, then the conclusions of Theorem~\ref{thm:introC} apply to both graph-based processes.
\end{enumerate}
\end{mainthm}

Part (1) of Theorem~\ref{thm:introD} is Theorem~\ref{t:rwa}; parts (2) and (3) combine Theorems~\ref{t:cl} and~\ref{t:cla} with Proposition~\ref{p:XF}. For the group-random-walk result, Gekhtman's theorem~\cite[Theorem~1.5]{Ge17} gives convergence of normalized \emph{oriented} closed-geodesic measures to a harmonic invariant measure in the boundary measure class $\check\lambda\times\lambda$, where $\lambda$ is the hitting measure for $\mu$ and $\check\lambda$ is the hitting measure for the reflected probability measure
\[
\check\mu(g):=\mu(g^{-1}).
\]
Symmetrization produces a usual flip-invariant current in the measure class
\[
(\check\lambda\times\lambda)+(\lambda\times\check\lambda).
\]
The semigroup hypothesis implies
\[
\supp(\lambda)=\supp(\check\lambda)=\partial F_N,
\]
so the symmetrized current has full support on $\partial^2F_N$ and is filling. For graph-based walks, the characteristic current is constructed directly from stationary block frequencies. Its filling property follows either from full support or from the $\R$-tree criterion of Proposition~\ref{p:ai}.

The standard simple non-backtracking walk on the rose is a special case: its characteristic current is the uniform current $\nu_A$, which has full support on $\partial^2F_N$. Thus the framework recovers the earlier generic model of \cite{KSS06}, while also allowing highly nonuniform and graph-directed random processes such as the biased process discussed above.

\subsection{Organization and further remarks}

Section~\ref{s:wh} recalls Whitehead's algorithm and peak reduction. Section~\ref{s:MLE} develops $M$-minimality, its quantitative refinements, algorithmic detection, and stabilizer bounds. Sections 4--6 treat currents, filling, and the neighborhood theorem underlying Theorem~\ref{thm:introB}. Section 7 proves the group-random-walk result, and Sections 8--9 develop the finite-state Markov-chain and graph-walk constructions.

The earlier results of \cite{KSS06} are strongly generic, with exponentially fast convergence of the relevant probabilities. The qualitative group-random-walk application here uses Gekhtman's equidistribution theorem followed by symmetrization, and the cited theorem does not provide a quantitative rate. A quantitative version would yield corresponding estimates after refining the definition of adaptation. In the graph-based setting, the frequency and cancellation arguments provide exponential or stretched-exponential estimates for their constituent events.

We are most grateful to Vadim Kaimanovich and Joseph Maher for many helpful discussions about random walks, for help with the references, and for clarifying several random-walk arguments. In particular, the proof of Proposition~\ref{p:supp} was explained to us by Kaimanovich. We are also grateful to the organizers of the March 2019 Dagstuhl conference ``Algorithmic Problems in Group Theory'' for providing impetus and motivation for completing this paper.


\section{Whitehead's algorithm}\label{s:wh}

Our main background reference for Whitehead's algorithm is Lyndon and Schupp, Chapter~I.4~\cite{LS}, and we refer the reader there for additional details.
Some other useful details and complexity results are available in \cite{KSS06,RVW}.
We recall the basic definitions and results here.

In this section we fix a free group  $F_N=F(A)$ of rank $N\ge 2$, with a fixed free basis $A=\{a_1,\dots, a_N\}$. Put $\Sigma_A=A\sqcup A^{-1}$.
We will also denote by $\mathcal C_N$ the set of all $F_N$-conjugacy classes $[g]$ where $g\in F_N$.

\begin{defn}[Whitehead automorphisms]\label{defn:moves}
  A \emph{Whitehead automorphism} of $F_N$ with respect to $A$ is an automorphism $\tau\in \Aut(F_N)$ of
  $F_N$ of one of the following two types:

  (1) There is a permutation $t$ of $\Sigma_A$ such that
  $\tau|_{\Sigma_A}=t$. In this case $\tau$ is called a \emph{relabeling
    automorphism} or a \emph{Whitehead automorphism of the first
    kind}.

  (2) There is an element $a\in \Sigma_A$, the \emph{multiplier}, such
  that for any $x\in \Sigma_A$
\[
\tau(x)\in \{x, xa, a^{-1}x, a^{-1}xa\}.
\]

In this case we say that $\tau$ is a \emph{Whitehead automorphism of
  the second kind}. (Note that since $\tau$ is an automorphism of $F_N$,
we always have $\tau(a)=a$ in this case).

We also refer to the images of Whitehead automorphisms in $\Out(F_N)$ as \emph{Whitehead moves} and sometimes again as \emph{Whitehead automorphisms}. We denote by $\mathcal W_N$ the set of all Whitehead moves $\tau\in\Out(F_N)$ such that $\tau\ne 1$ in $\Out(F_N)$.
\end{defn}

Note that for any $a\in \Sigma_A$ the inner automorphism $\operatorname{ad}(a)\in\Aut(F_N)$ is a
Whitehead automorphism of the second kind. Note also that if $\tau\in \mathcal W_N$ then $\tau^{-1}\in \mathcal W_N$.

To simplify the exposition, we formulate all the definitions and results related to Whitehead's algorithm in terms of conjugacy classes of elements of $F_N$.
In this context we usually think of an input $[w]\in \mathcal C_N$ as given by a cyclically reduced word $w\in F(A)$.  Since for $w\in F_N$ we have $||w||_A\le |w|_A$, and since it takes linear time in $|w|_A$ to find a cyclically reduced form of $w\in F(A)$ (see \cite{KSS06} for additional discussion on this topic), all our complexity estimates hold in terms of $|w|_A$.

\begin{defn}[Minimal and Whitehead-minimal elements]
  A conjugacy class $[w]\in\mathcal C_N$ is \emph{$\Out(F_N)$-minimal} with respect to $A$ if for every $\phi\in\Out(F_N)$ we have $||w||_A\le ||\phi(w)||_A$.

  A conjugacy class $[w]\in\mathcal C_N$ is \emph{Whitehead-minimal} with respect to $A$ if for every Whitehead move $\tau\in\mathcal W_N$ we have $||w||_A\le ||\tau(w)||_A$.

  For $[w]\in \mathcal C_N$, denote $\mathcal M([w])=\{[u]\in \Out(F_N)[w]| [u] \text{ is $\Out(F_N)$-minimal}\}$. 
\end{defn}
Note that, by definition, an $\Out(F_N)$-minimal $[w]$ is necessarily Whitehead-minimal.

\begin{defn}[Automorphism graph]\label{d:autgraph}
The \emph{automorphism graph} of $F_N$ is the labelled multigraph $\mathcal T$ defined as follows. Its vertex set is
\[
V\mathcal T=\mathcal C_N.
\]
For every $[w]\in\mathcal C_N$ and every $\tau\in\mathcal W_N$ satisfying
$\|\tau(w)\|_A=\|w\|_A$, introduce an oriented edge
\[
e([w],\tau):[w]\longrightarrow \tau([w])
\]
labelled by $\tau$, and equip the edge set with the involution
\[
e([w],\tau)^{-1}=e(\tau([w]),\tau^{-1}).
\]
The resulting underlying multigraph allows parallel edges and loops. If the displayed involution fixes the move data, the corresponding loop still has its two formal orientations, both carrying the same involutive label. Thus distinct Whitehead moves inducing the same map between conjugacy classes remain distinct labelled edges.

For $n\ge 0$, let $\mathcal T_n$ be the full subgraph spanned by the vertices $[w]$ with $\|w\|_A=n$. For a vertex $[w]\in V\mathcal T_n$, denote by $\mathcal T_n[w]$ the connected component of $\mathcal T_n$ containing $[w]$.
\end{defn}

We first state the following simplified version of Whitehead's ``peak reduction'' lemma (see \cite[Proposition~1.2]{KSS06}):

\begin{prop}\label{p:wpr}
Let $N\ge 2$ be an integer. Then the following hold:

\begin{enumerate}
\item An element $[w]\in \mathcal C_N$ is $\Out(F_N)$-minimal if and only if $[w]$ is Whitehead-minimal. (Thus if $[w]$ is not $\Out(F_N)$-minimal then there exists $\tau\in\mathcal W_N$ such that $||\tau(w)||_A<||w||_A$).
\item Suppose that $[w]\ne [w']$ are both $\Out(F_N)$-minimal. Then $\Out(F_N)[w]=\Out(F_N)[w']$ if and only if $||w||_A=||w'||_A=n\ge 0$, and there exists a finite sequence $\tau_1,\dots\tau_k\in\mathcal W_N$ such that $\tau_k\dots \tau_1[w]=[w']$ and that for $i=1,\dots, k$ we have
\[
||\tau_i\dots \tau_1(w)||_A=n.
\]
\end{enumerate}
\end{prop}

Proposition~\ref{p:wpr} implies that if $[w]\in \mathcal C_N$ is $\Out(F_N)$-minimal with $||w||_A=n$ then $\mathcal M([w])=V\mathcal T_n[w]$.

We also record the following more general version of ``peak reduction'':

\begin{prop}\label{p:wprs}\cite[Proposition~4.17]{LS}
Let $[w],[w']\in \mathcal C_N$ and $\phi\in\Out(F_N)$ be such that $[w']=\phi([w])$ and $\|w'\|_A\le \|w\|_A$.
Then there exists a factorization $\phi=\tau_k\dots \tau_1$ in $\Out(F_N)$, where $\tau_i\in \mathcal W_N$ and where $\|\tau_i\dots\tau_1(w)\|_A\le \|w\|_A$ for $i=1,\dots,k$.
\end{prop}

\begin{defn}[Whitehead algorithm]
Let $F_N=F(A)$ be free of rank $N\ge 2$, with a fixed free basis $A$.

$\bullet$ The \emph{Whitehead minimization algorithm} is the following process. Given $[w]\in \mathcal C_N$ put $[w_1]=[w]$. If $[w_i]$ is already constructed, check if there exists $\tau\in \mathcal W_N$ such that $||\tau(w_i)||_A<||w_i||_A$. If not, declare that $[w_i]\in\mathcal M([w])$ (that is $[w_i]$ is an $\Out(F_N)$-minimal element in $\Out(F_N)[w]$ and terminate the algorithm. Put $[w_{i+1}]=[\tau(w_i)]$.

$\bullet$ The \emph{Whitehead stabilization algorithm} is the following process. Suppose that $[w]\in \mathcal C_N$ is Whitehead-minimal (and therefore $\Out(F_N)$-minimal) with $||w||_A=n\ge 0$. Construct the component $\mathcal T_n[w]$ of $\mathcal T_n$ using the ``breadth-first'' stabilization process. Start with $S_1=\{[w]\}$. Now if a finite collection $S_i$  of conjugacy classes with $||.||_A=n$ is already constructed, for each element $[u]\in S_i$ and each $\tau\in \mathcal W_N$, put
\[
S_{i+1}=S_i\cup\{\tau([u])| [u]\in S_i, \tau\in \mathcal W_N \text{ and } ||\tau(u)||_A=n\}.
\]
Terminate the process with the output $S_i$ for the smallest $i\ge 1$ such that $S_{i+1}=S_i$. Declare that $S_i=V\mathcal T_n[w]=\mathcal M([w])$.

$\bullet$ The \emph{Whitehead algorithm} is the following process. Given $[w], [w']\in \mathcal C_N$, first apply the Whitehead minimization process to each of $[w], [w']$ to output elements $[u], [u']$ accordingly. Declare that $[u]\in \mathcal M([w])$ and $[u']\in \mathcal M([w'])$. If $||u||_A\ne ||u'||_A$, declare that $\Out(F_N)[w]\ne \Out(F_N)[w']$ and terminate the process. If $||u||_A=||u'||_A=0$, then $[u]=[u']=[1]$; declare the inputs equivalent and terminate. Thus suppose that $||u||_A=||u'||_A=n\ge1$. Apply the Whitehead stabilization algorithm to $[u]$ to produce the set $S$. Declare that $S=V\mathcal T_n[u]=\mathcal M([w])$. Then check whether $[u']\in S$. If $[u']\in S$, declare that $\Out(F_N)[w]=\Out(F_N)[w']$, and if $[u']\not\in S$, declare that $\Out(F_N)[w]\ne\Out(F_N)[w']$, and terminate the process.

\end{defn}

\begin{rem}\label{rem:com}
Part (1) of Proposition~\ref{p:wpr} implies that the Whitehead minimization algorithm on an input $[w]\in \mathcal C_N$  always terminates in $O(|w|_A^2)$ time and indeed outputs an element of $\mathcal M([w])$.  The quadratic time bound arises since going from $[w_i]$ to $[w_{i+1}]$ takes a priori linear time in $||w_i||_A$, and since $||w_1||_A>||w_2||_A>\dots$, the process terminates with some $[w_k]$ such that $k\le ||w||_A$.

Part (2) of Proposition~\ref{p:wpr} implies that the Whitehead stabilization algorithm on an $\Out(F_N)$-minimal input $[w]$ with $\|w\|_A=n$ outputs the set $\mathcal M([w])=V\mathcal T_n[w]$. In the standard word model, exploring a component with $V$ vertices takes $O(nV)$ time, since one tests only the fixed finite set $\mathcal W_N$ at each vertex and each test takes $O(n)$ time. Taken together, Proposition~\ref{p:wpr} implies that Whitehead's algorithm correctly decides whether or not $\Out(F_N)[w]=\Out(F_N)[w']$.

Overall, the a priori worst-case complexity of Whitehead's algorithm on the input $[w],[w']$ is exponential in $\max\{|w|_A, |w'|_A\}$ because for $[u]\in V\mathcal T_n$ the cardinality $\#V\mathcal T_n[u]$ is at most exponential in $n$.
\end{rem}

\begin{defn}
Suppose that $\mathfrak W\subseteq \Out(F_N)$ is a fixed finite set of \emph{auxiliary} automorphisms.
\begin{itemize}
\item The \emph{$\mathfrak W$-speed-up} of the Whitehead minimization algorithm takes an input $[w]\ne[1]$, computes
\[
\mathfrak W[w]=\{\psi[w]:\psi\in\mathfrak W\},
\]
and runs the Whitehead minimization algorithm in parallel on $[w]$ and on every element of $\mathfrak W[w]$, stopping when the first branch terminates. Since all branches remain in the orbit $\Out(F_N)[w]$, the output belongs to $\mathcal M([w])$.
\item The \emph{$\mathfrak W$-speed-up} of Whitehead's algorithm first applies this speeded-up minimization procedure to two nontrivial inputs $[w]$ and $[w']$, producing $[u]\in\mathcal M([w])$ and $[u']\in\mathcal M([w'])$, and then applies the usual stabilization procedure to decide whether $\Out(F_N)[w]=\Out(F_N)[w']$.
\end{itemize}
\end{defn}
Since $\mathfrak W$ is finite and fixed, the parallel branches may be implemented by a round-robin simulation with only a constant multiplicative overhead. Thus the a priori complexity estimates remain those of Remark~\ref{rem:com}, up to constants depending on $\mathfrak W$. When $\mathfrak W=\{\psi\}$, we also call this the $\psi$-speed-up.

\section{\texorpdfstring{$M$-minimality and Whitehead's algorithm}{M-minimality and Whitehead's algorithm}}\label{s:MLE}

Let $F_N=F(A)$ be free of rank $N\ge 2$  where $A=\{a_1,\dots, a_N\}$ is a fixed free basis of $F_N$.

\subsection{Main definitions}

\begin{defn}[$M$-minimal elements]\label{d:M}
Let $M\ge 1$ be an integer. A conjugacy class $[u]\in \mathcal C_N$ is called \emph{$M$-minimal} if the following condition holds:

Whenever  $\tau_1,\tau_2,\dots \tau_k\in \mathcal W_N$ are such that for $[u_i]=\tau_i\dots \tau_1([u])$, where $i=1,\dots, k$, we have
\[
||u||_A\ge ||u_1||_A\ge ||u_2||_A\ge \dots \ge  ||u_k||_A
\]
and that the conjugacy classes $[u],[u_1],\dots, [u_k]\in \mathcal C_N$ are distinct, then $k\le M$.

\end{defn}

\begin{lem}\label{lem:M}
Let $M\ge 1$ be an integer and let $[u]\in\mathcal C_N$ be $M$-minimal. Put $B=\#\mathcal W_N$ and
\[
M'=M'(M,N):=\sum_{j=0}^{M}B^j.
\]
Then the following hold.
\begin{enumerate}
\item If $\tau\in\mathcal W_N$ and $\|\tau(u)\|_A<\|u\|_A$, then $\tau([u])$ is $M$-minimal. If $\|\tau(u)\|_A=\|u\|_A$, then $\tau([u])$ is $2M$-minimal.
\item We have $\#\mathcal M([u])\le M'$.
\item Every $[u']\in\mathcal M([u])$ is $(M'-1)$-minimal.
\item Every strictly length-decreasing chain of Whitehead moves starting at $[u]$ has length at most $M$.
\end{enumerate}
\end{lem}
\begin{proof}
Let $[v]=\tau([u])$ with $\|v\|_A<\|u\|_A$. Any simple non-increasing chain starting at $[v]$ can be prefixed by the step $[u]\to[v]$; because all subsequent lengths are strictly smaller than $\|u\|_A$, the class $[u]$ cannot reappear. Hence the prefixed chain has length at most $M$, and $[v]$ is $M$-minimal.

Now suppose that $\|v\|_A=\|u\|_A$, and consider a simple non-increasing chain of length $k$ starting at $[v]$. If the chain does not contain $[u]$, prefixing it by $[u]\to[v]$ gives $k+1\le M$. If it first reaches $[u]$ after $j$ steps, all lengths along the initial segment are equal. Reversing that segment gives a simple non-increasing chain of length $j$ starting at $[u]$, while the remaining segment gives one of length $k-j$ starting at $[u]$. Thus $j\le M$ and $k-j\le M$, so $k\le 2M$.

To prove (2), first apply the Whitehead minimization algorithm to $[u]$, deleting repetitions if necessary, and obtain an orbit-minimal class $[v]$ by a strictly decreasing chain. By $M$-minimality this chain has length at most $M$. For any $[u']\in\mathcal M([u])$, Proposition~\ref{p:wpr}(2) gives a simple length-preserving chain from $[v]$ to $[u']$. Concatenating the two chains produces a simple non-increasing chain from $[u]$ to $[u']$; the two parts meet only at $[v]$, since every earlier vertex has length greater than the minimum. Hence the total length is at most $M$. There are at most $\sum_{j=0}^{M}B^j=M'$ possible endpoints of such labelled chains, proving (2).

If $[u']\in\mathcal M([u])$, every non-increasing chain starting at $[u']$ is length-preserving and remains inside $\mathcal M([u])$. A simple such chain has at most $M'$ vertices, and therefore length at most $M'-1$. This proves (3). Part (4) is immediate from Definition~\ref{d:M}.
\end{proof}

Recall that as defined in \cite{KSS06}, an element $[u]\in \mathcal C_N$ is called \emph{strictly minimal} if for every Whitehead automorphism of the second kind $\tau\in \mathcal W_N$ we have $||u||_A<||\tau(u)||_A$.

\begin{lem}\label{lem:sm}
Put $M=2^N N!$. Let $[u]\in \mathcal C_N$ be strictly minimal. Then $[u]$ is $M$-minimal.
\end{lem}
\begin{proof}
Let $[u]\in \mathcal C_N$ be strictly minimal. Let $\tau_1,\dots, \tau_k\in \mathcal W_N$ be such that for $[u_i]=\tau_i\dots\tau_1([u])$ with $i=1,\dots, k$ we have
\[
||u||_A\ge ||u_1||_A\ge \dots \ge ||u_k||_A
\]
and that the conjugacy classes $[u],[u_1],\dots, [u_k]\in \mathcal C_N$ are distinct.
As shown in \cite{KSS06}, in this case each $\tau_i$ is a Whitehead automorphism of the first kind and each $u_i$ is again strictly minimal, for $i=1,\dots, k$. Hence $\phi_i=\tau_i\dots\tau_1\in \Out(F_N)$ is a re-labelling automorphism for $i=1,\dots, k$. The number of re-labelling automorphisms of $F_N$ is equal to $2^N N!$. Since $[u_1],\dots, [u_k]\in \mathcal C_N$ are distinct, it follows that $\phi_1,\dots, \phi_k\in\Out(F_N)$ are distinct, and hence $k\le 2^N N!$. Thus $[u]$ is $M$-minimal for $M=2^N N!$, as required.

\end{proof}

In the context of this paper the sources of $M$-minimal elements (various random processes) provide elements with additional, sharper, metric properties, captured by the definitions below.

\begin{defn}\label{d:MLE}
Let $M\ge 1$ be an integer, let $\lambda>1$ and let $0< \epsilon<\lambda-1$.

A finite set $S$ of conjugacy classes of nontrivial elements of $F_N$ is called \emph{$(M,\lambda,\epsilon)$-minimizing} if it satisfies the following properties:

\begin{enumerate}
\item We have $\#(S)\le M$.
\item For any $[u],[u']\in S$ we have $\Out(F_N)[u]=\Out(F_N)[u']$.
\item For any $[u],[u']\in S$ we have $1-\epsilon\le \frac{||u'||_A}{||u||_A}\le 1+\epsilon$.
\item For every $[u]\in S$ and every $\phi\in \Out(F_N)$ such that $\phi([u])\not\in S$ we have $\frac{||\phi(u)||_A}{||u||_A}\ge \lambda>1+\epsilon$.
\end{enumerate}

In this case for any $[u]\in S$ we also say that $S$ is a \emph{$(M,\lambda,\epsilon)$-minimizing set for $[u]$}.

We say that a nontrivial conjugacy class $[u]$ in $F_N$ is \emph{$(M,\lambda,\epsilon)$-minimal} if there exists an $(M,\lambda,\epsilon)$-minimizing set $S$ for $[u]$  (and thus $[u]\in S$).
\end{defn}

Note that if $S$ is a $(M,\lambda,\epsilon)$-minimizing set and if $[u]\in S$ then for $\phi\in \Out(F_N)$ either $\phi(u)\in S$ or $\frac{||\phi(u)||_A}{||u||_A}\ge \lambda$, and these outcomes are mutually exclusive.

\begin{defn}\label{d:MLEW}
Let $M\ge 1$ be an integer, let $\lambda>1$ and let $0< \epsilon<\lambda-1$.

A finite set $S\subseteq \mathcal C_N$ of conjugacy classes of nontrivial elements of $F_N$ is called \emph{$(M,\lambda,\epsilon, \mathcal W_N)$-minimizing} if it satisfies the following properties:

\begin{enumerate}
\item We have $\#(S)\le M$.
\item For any $[u],[u']\in S$ we have $\Out(F_N)[u]=\Out(F_N)[u']$.
\item For any $[u],[u']\in S$ we have $1-\epsilon\le \frac{||u'||_A}{||u||_A}\le 1+\epsilon$.
\item For any $[u]\in S$ and $\tau\in \mathcal W_N$ exactly one of the following occurs:
\begin{itemize}
\item[(i)] We have $\tau([u])\in S$.
\item[(ii)] We have $\tau([u])\not\in S$ and $\frac{||\tau(u)||_A}{||u||_A}\ge \lambda>1+\epsilon$.
\end{itemize}
\end{enumerate}

In this case for any $[u]\in S$ we also say that $S$ is a \emph{$(M,\lambda,\epsilon, \mathcal W_N)$-minimizing set for $[u]$}.

We say that a nontrivial conjugacy class $[u]$ in $F_N$ is \emph{$(M,\lambda,\epsilon, \mathcal W_N)$-minimal} if there exists an $(M,\lambda,\epsilon, \mathcal W_N)$-minimizing set $S$ for $[u]$  (and thus $[u]\in S$).

\end{defn}

\begin{lem}\label{lem:aux}
Let $M\ge 1$ be an integer, let $\lambda>1$, $0<\epsilon<1$ be such that $\epsilon<\lambda-1$ and $\lambda (1-\epsilon)>1$. Let $S\subseteq \mathcal C_N$ be a finite set of conjugacy classes of nontrivial elements of $F_N$ such that $S$ is $(M,\lambda,\epsilon, \mathcal W_N)$-minimizing.
\begin{enumerate}
\item For any $[u]\in S$ and $\tau\in\mathcal W_N$ such that $||\tau(u)||_A\le (1+\epsilon) ||u||_A$ we have $\tau([u])\in S$.
\item For any $[u]\in S$ and $\phi\in\Out(F_N)$ such that $||\phi(u)||_A\le ||u||_A$ we have $\phi([u])\in S$.
\item For any $[u]\in S$ we have $\mathcal M([u])\subseteq S$. 
\end{enumerate}
\end{lem}
\begin{proof}
Part (1) follows from conditions (3), (4) of Definition~\ref{d:MLEW}.

For (2), suppose that  $[u]\in S$ and $\phi\in\Out(F_N)$ are such that $||\phi(u)||_A\le ||u||_A$. By Proposition~\ref{p:wprs}, there exist $\tau_1,\dots,\tau_k\in \mathcal W_N$ such that $\phi=\tau_k\dots\tau_1$ and that for $[u_0]=[u]$, $[u_i]=\tau_i\dots \tau_1([u])$ for $i=1,\dots, k$ we have $||u_i||_A\le ||u||_A$. Note that $[u_k]=\phi([u])$.

We argue by induction on $i$ that $[u_i]\in S$ for $i=1,\dots, k$.
We have $[u]=[u_0]\in S$. Suppose now $0\le i<k$ and $[u_i]\in S$.  We need to show that $[u_{i+1}]=\tau_{i+1}[u_i]\in S$. Suppose, on the contrary, that $[u_{i+1}]\not\in S$. Then $||u_{i+1}||_A/||u_i||_A\ge \lambda$. Since $[u], [u_i]\in S$, also have $||u_i||_A/||u||_A\ge 1-\epsilon$. Therefore $||u_{i+1}||_A/||u||_A\ge \lambda(1-\epsilon)$, so that $||u_{i+1}||_A\ge \lambda(1-\epsilon)||u||_A>||u||_A$ since $\lambda(1-\epsilon)>1$. This contradicts the choice of $\tau_1,\dots, \tau_k$.  Thus $[u_{i+1}]\in S$, as required.

Hence $[u_k]=\phi([u])\in S$, and part (2) of the lemma is verified.

Now part (2) directly implies part (3).
\end{proof}

The following result plays a key role in this paper:

\begin{prop}\label{prop:M}
Let $M\ge1$, let $\lambda>1$, and let $0<\epsilon<\lambda-1$. If $[u]\in\mathcal C_N$ is $(M,\lambda,\epsilon,\mathcal W_N)$-minimal, then $[u]$ is $M$-minimal.
\end{prop}
\begin{proof}
Let $S\subseteq \mathcal C_N$ be an $(M,\lambda,\epsilon,\mathcal W_N)$-minimizing set containing $[u]$.  

Suppose now that $\tau_1,\tau_2,\dots \tau_k\in \mathcal W_N$ are such that for $[u_i]=\tau_i\dots\tau_1([u])$, where $i=1,\dots, k$, we have
\[
||u||_A\ge ||u_1||_A\ge \dots \ge ||u_i||_A\ge \dots  \ge ||u_k||_A,
\]
and that the conjugacy classes $[u], [u_1],\dots, [u_k]$ are distinct.

We prove inductively that $[u_i]\in S$. If $[u_{i-1}]\in S$ but $[u_i]\notin S$, Definition~\ref{d:MLEW}(4)(ii) would give
\[
\|u_i\|_A\ge\lambda\|u_{i-1}\|_A>\|u_{i-1}\|_A,
\]
contrary to the assumed non-increase of lengths. Thus all $[u_i]$ lie in $S$. Since the conjugacy classes $[u],[u_1],\dots,[u_k]$ are distinct and $\#S\le M$, it follows that $k\le M-1$. Hence $[u]$ is $M$-minimal.
\end{proof}

\begin{rem}\label{r:cont}
The definitions directly imply that every $(M,\lambda,\epsilon)$-minimizing set is $(M,\lambda,\epsilon,\mathcal W_N)$-minimizing. Consequently, every $(M,\lambda,\epsilon)$-minimal conjugacy class is $(M,\lambda,\epsilon,\mathcal W_N)$-minimal.
\end{rem}

The following partial converse holds after slightly strengthening the parameters.

\begin{prop}\label{prop:cont}
Let $M\ge1$, let $\lambda>1$, and let $0<\epsilon<\lambda-1$. Let $0<\epsilon'<\epsilon$ and $\lambda'>\lambda$ satisfy
\[
\lambda'(1-\epsilon')>\lambda.
\]
If $S\subseteq\mathcal C_N$ is $(M,\lambda',\epsilon',\mathcal W_N)$-minimizing, then $S$ is $(M,\lambda,\epsilon)$-minimizing.
\end{prop}
\begin{proof}
Conditions (1)--(3) of Definition~\ref{d:MLE} follow immediately. Let $[u]\in S$ and let $\phi\in\Out(F_N)$ satisfy $\phi([u])\notin S$. Lemma~\ref{lem:aux}(2), applied with $\lambda',\epsilon'$, implies that $\|\phi(u)\|_A>\|u\|_A$.

Apply Proposition~\ref{p:wprs} to $\phi^{-1}$ and the pair $\phi([u]),[u]$. Reversing the resulting factorization gives
\[
\phi=\tau_k\cdots\tau_1
\]
with $\tau_i\in\mathcal W_N$ such that, for $[v_0]=[u]$ and $[v_i]=\tau_i\cdots\tau_1([u])$,
\[
\|v_i\|_A\le \|\phi(u)\|_A\qquad(0\le i\le k).
\]
Let $j\ge1$ be the first index for which $[v_j]\notin S$. Then $[v_{j-1}]\in S$, and Definition~\ref{d:MLEW}(4) gives
\[
\|v_j\|_A\ge\lambda'\|v_{j-1}\|_A.
\]
Since $[v_{j-1}],[u]\in S$, condition (3) gives $\|v_{j-1}\|_A\ge(1-\epsilon')\|u\|_A$. Therefore
\[
\|\phi(u)\|_A\ge\|v_j\|_A
 \ge\lambda'(1-\epsilon')\|u\|_A
 >\lambda\|u\|_A.
\]
Thus condition (4) of Definition~\ref{d:MLE} holds.
\end{proof}

\subsection{Behavior of Whitehead's algorithm}

\begin{prop}\label{p:MLE}
Let $M\ge 1$ be an integer and let $[u]\in \mathcal C_N$ be $M$-minimal.  Then the following hold:

\begin{enumerate}
\item If $\tau_1,\tau_2,\dots \tau_k\in \mathcal W_N$ are such that
\[
||u||_A>||\tau_1(u)||_A>||\tau_2\tau_1(u)||_A>\dots > ||\tau_k\dots \tau_2\tau_1(u)||_A
\]
then $k\le M$.

\item If a sequence $\tau_1,\tau_2,\dots,\tau_k\in \mathcal W_N$ as in (1) has the property that $[u_k]$ is $\mathcal W_N$-minimal then $[u_k]\in \mathcal M([u])$.

\item For any $[u']\in \mathcal M([u])$ we have $||u||_A\le 3^M ||u'||_A$.
\end{enumerate}
\end{prop}
\begin{proof}
Part (1) of the proposition is exactly part (4) of Lemma~\ref{lem:M}. Part~(2) holds by the general peak reduction properties of Whitehead's algorithm given in Proposition~\ref{p:wpr}.  

For part (3),  let $\tau_1,\tau_2,\dots \tau_k\in \mathcal W_N$ be as in (1) such that $[u_k]$ is $\mathcal W_N$-minimal. Then $[u_k]\in \mathcal M([u])$ so that $||u_k||_A=||u'||_A$. Note that each $\tau_i^{-1}$ is again an element of $\mathcal W_N$. For any Whitehead automorphism $\tau\in \mathcal W_N$ and any $w\in F_N$ we have $||\tau(w)||_A\le 3||w||_A$. Therefore $||u||_A\le 3^k||u_k||_A=3^k||u'||_A\le 3^M||u'||_A$.  Thus part (3) of the proposition holds.
\end{proof}

Remark~\ref{r:cont} and Proposition~\ref{prop:M} imply that the conclusions of Proposition~\ref{p:MLE} also apply to $(M,\lambda,\epsilon)$-minimal conjugacy classes.

\begin{defn}
Let $M\ge 1$ be an integer.

\begin{enumerate}
\item We denote by $U_N(M)$ the set of all $1\ne u\in F_N$ such that $[u]$ is  $M$-minimal.
\item We denote by $Y_N(M)$ the set of all $1\ne w\in F_N$ such that there exists $[u]\in \Out(F_N)[w]$ such that $[u]$ is  $M$-minimal. 


\item Let $\psi\in \Out(F_N)$. Denote by  $U_N(M,\psi)$ the set of all $1\ne w\in F_N$ such that $\psi([w])$ is  $M$-minimal.

\end{enumerate}
\end{defn}

\begin{lem}\label{lem:YN}
Let $M\ge1$, and let $M'=M'(M,N)$ be as in Lemma~\ref{lem:M}. If $u\in Y_N(M)$, then
\[
\#\mathcal M([u])\le M',
\]
and every element of $\mathcal M([u])$ is $(M'-1)$-minimal.
\end{lem}
\begin{proof}
Choose an $M$-minimal conjugacy class $[v]\in\Out(F_N)[u]$. Then $\mathcal M([u])=\mathcal M([v])$, and the result follows from Lemma~\ref{lem:M}(2),(3).
\end{proof}

We now summarize algorithmic properties of  $M$-minimal elements in relation to Whitehead's algorithm.

\begin{thm}\label{t:WHM} Let $M\ge 1$ be an integer.  Then there exists a constant $K\ge 1$ such that the following hold:
\begin{enumerate}
\item[(a)] For any $u\in U_N(M)$ the Whitehead minimization algorithm on the input $u$ terminates in time at most $K|u|_A$ and produces an element of $\mathcal M([u])$.

\item[(b)] For any $u_1,u_2\in U_N(M)$, the Whitehead algorithm for the automorphic equivalence problem in $F_N$ terminates in time at most $K\max \{|u_1|_A,|u_2|_A\}$, on the input $(u_1,u_2)$.
\item[(c)] For any $u_1\in Y_N(M)$ and any $1\ne u_2\in F_N$, the Whitehead algorithm for the automorphic equivalence problem in $F_N$ terminates in time at most $K\max \{|u_1|_A^2,|u_2|_A^2\}$, on the input $(u_1,u_2)$.
\item[(d)] For any $u_1\in U_N(M)$ and any $1\ne u_2\in F_N$, the Whitehead algorithm for the automorphic equivalence problem in $F_N$ terminates in time at most $K\max \{|u_1|_A,|u_2|_A^2\}$, on the input $(u_1,u_2)$.
\item[(e)] Let $\psi\in\Out(F_N)$ be a fixed element.  Then there is a constant $K'=K'(N,M,\psi)\ge1$ such that for any $u_1,u_2\in U_N(M,\psi)$, the $\psi$-speed-up of Whitehead's algorithm decides in time at most $K'\max \{|u_1|_A,|u_2|_A\}$, whether or not $\Out(F_N)[u_1]=\Out(F_N)[u_2]$.

\item[(f)] Let $\psi\in\Out(F_N)$ be a fixed element. Then there is a constant $K'=K'(N,M,\psi)\ge1$ such that for any $u_1\in U_N(M,\psi)$ and any $1\ne u_2\in F_N$, the $\psi$-speed-up of Whitehead's algorithm decides
in time at most $K'\max \{|u_1|_A,|u_2|_A^2\}$, whether or not $\Out(F_N)[u_1]=\Out(F_N)[u_2]$.
\end{enumerate}

\end{thm}

\begin{proof}

Let $M'=M'(M,N)\ge 1$ be the constant provided by part (2) of Lemma~\ref{lem:M}.

(a) Let $u\in U_N(M)$ be arbitrary.  Proposition~\ref{p:MLE} implies that starting with $u$ and iteratively looking for Whitehead moves that decrease the $||.||_A$-length terminates after a chain of $\le M$ such moves with a conjugacy class that is Whitehead-minimal and therefore is $\Out(F_N)$-minimal, that is, an element of $\mathcal M([u])$. This process takes at most time $C_1|u|_A$ for some constant $C_1>0$ depending only on $N,M$.

(b) Let $u_1,u_2\in U_N(M)$ so that $[u_1],[u_2]$ are $M$-minimal. By part (a) above,  applying the Whitehead minimization algorithm to $[u_i]$ terminates in time $\le C_1|u_i|_A$  with an $\Out(F_N)$-minimal element $[u_i']\in \mathcal M([u_i])$ such that $n_i=||u_i'||_A\le ||u_i||_A\le |u_i|_A$. If $n_1\ne n_2$ then $\Out(F_N)[u_1]\ne \Out(F_N)[u_2]$ and we are done. Suppose that $n=n_1=n_2$. 

By Proposition~\ref{p:wpr} we have $\mathcal M([u_i])=V\mathcal T_n[u_i']$ for $i=1,2$. Moreover, by part (2) of Lemma~\ref{lem:M} we have $\#V\mathcal T_n[u_i']\le M'$ here. Since $M, M'$ are fixed, it takes linear time in $|u_i|_A$ to construct the graph $\mathcal T_n[u_i']$ from $u_i'$. Then $\Out(F_N)[u_1]=\Out(F_N)[u_2]$ if and only if $[u_2']\in V\mathcal T_n[u_1']$, and this condition can be checked in linear time in $\max\{|u_1|_A,|u_2|_A\}$. Summing up we get that the total running time of the Whitehead algorithm for the automorphic equivalence problem in $F_N$ is time at most $C_2\max \{|u_1|_A,|u_2|_A\}$, for some constant $C_2>0$ depending only on $N,M$.

(c) Now suppose that $u_1\in Y_N(M)$ and $1\ne u_2\in F_N$. We first apply the Whitehead minimization algorithm to each of $u_1,u_2$ to find $\Out(F_N)$-minimal elements $[u_i']\in \Out(F_N)[u_i]$ for $i=1,2$. Producing $u_i'$ from $u_i$ takes quadratic time in terms of $|u_i|_A$. By Lemma~\ref{lem:YN}, the set $\mathcal M([u_1])=\mathcal M([u_1'])$ has cardinality at most $M'$.

Put $n_i=\|u_i'\|_A$. If $n_1\ne n_2$ then $\Out(F_N)[u_1]\ne \Out(F_N)[u_2]$ and we are done. Suppose that $n=n_1=n_2$. Since $u_1'$ is $\Out(F_N)$-minimal, Proposition~\ref{p:wpr} gives $\mathcal M([u_1])=\mathcal M([u_1'])=V\mathcal T_n[u_1']$, and Lemma~\ref{lem:YN} gives $\#V\mathcal T_n[u_1']\le M'$. Then, since $M, M'$ are fixed, it takes at most linear time in $n=||u_1'||_A\le |u_1|_A$ to construct the graph $\mathcal T_n[u_1']$. Recall also that $[u_2']\in \mathcal M([u_2])$ and $\|u_2'\|_A=n$. Then we have $\Out(F_N)[u_1]=\Out(F_N)[u_2]$ if and only if $[u_2']\in V\mathcal T_n[u_1']$.  This last condition can be checked in linear time in $n$. Again, summing up we see that the total running time of the Whitehead algorithm on $(u_1,u_2)$ is at most $C_3\max \{|u_1|_A^2,|u_2|_A^2\}$, for some constant $C_3>0$ depending only on $N,M$.

(d) Now let $u_1\in U_N(M)$ and $1\ne u_2\in F_N$. We first apply the Whitehead minimization algorithm to each of $u_1,u_2$ to find $\Out(F_N)$-minimal elements $[u_i']\in \Out(F_N)[u_i]$ for $i=1,2$. As in (b), producing $u_1'$ from $u_1$ takes linear time in $|u_1|_A$, because $u_1$ is $M$-minimal. Producing $u_2'$ from $u_2$ takes at most quadratic time in $|u_2|_A$, by the general Whitehead's minimization algorithm properties. After that we proceed exactly in (c) above to decide if $[u_1']$ and $[u_2']$ are $\Out(F_N)$-equivalent.  Summing up we see that the total running time of the Whitehead algorithm on $(u_1,u_2)$ is at most $C_4\max \{|u_1|_A,|u_2|_A^2\}$ in this case, for some constant $C_4>0$ depending only on $N,M$.

(e) Choose a representative $\Psi\in\Aut(F_N)$ of $\psi$, and choose $C_\psi\ge1$ such that
\[
|\Psi(g)|_A\le C_\psi |g|_A
\qquad(g\in F_N).
\]
Let $u_i\in U_N(M,\psi)$. In the $\psi$-speed-up of minimization, the branch beginning at $\Psi(u_i)$ begins with an $M$-minimal conjugacy class and has input length at most $C_\psi|u_i|_A$. By part (a), that branch terminates in time $O(|u_i|_A)$. The round-robin implementation therefore makes the entire speeded-up minimization terminate in linear time. Let $[v_i]\in\mathcal M([u_i])$ be its output.

The orbit $\Out(F_N)[u_i]$ contains the $M$-minimal class $\psi([u_i])$, so Lemma~\ref{lem:YN} gives
\[
\#\mathcal M([u_i])\le M'.
\]
Moreover, with $C_0=\max\{1,C_\psi\}$, one has $\|v_i\|_A\le C_0|u_i|_A$. If $\|v_1\|_A\ne\|v_2\|_A$, the speeded-up Whitehead algorithm stops after minimization. If their common length is $n$, then
\[
V\mathcal T_n[v_1]=\mathcal M([u_1])
\]
has at most $M'$ vertices. Its breadth-first construction and the final membership test take $O(nM')$ time. Since $M'$ is fixed, the total running time is linear in $\max\{|u_1|_A,|u_2|_A\}$, proving (e).

(f) Apply the same speeded-up minimization to $u_1\in U_N(M,\psi)$; it terminates in $O(|u_1|_A)$ time and produces $[v_1]\in\mathcal M([u_1])$. On the arbitrary second input $u_2$, the original branch of the speed-up is the ordinary Whitehead minimization algorithm, so the whole round-robin procedure terminates in $O(|u_2|_A^2)$ time and produces $[v_2]\in\mathcal M([u_2])$. If the two minimal lengths agree, the component $V\mathcal T_n[v_1]=\mathcal M([u_1])$ still has at most $M'$ vertices, since the orbit of $u_1$ contains the $M$-minimal class $\psi([u_1])$. Thus stabilization takes only linear time in the common minimal length. The total running time is therefore
\[
O\bigl(\max\{|u_1|_A,|u_2|_A^2\}\bigr),
\]
which proves (f).

\end{proof}

We recall another basic fact related to Whitehead's algorithm which describes $\Out(F_N)$-stabilizers of conjugacy classes in $F_N$:

\begin{prop}\label{prop:LSST}\cite[Proposition~8.1]{KSS06}
Let $1\ne u\in F_N$ be such that $[u]$ is $\Out(F_N)$-minimal, and let $n=||u||_A$. Then for $\phi\in\Out(F_N)$ we have $\phi([u])=[u]$ if and only if there exists a sequence of Whitehead automorphisms $\tau_1,\dots, \tau_k\in \mathcal W_N$ such that for $[u_i]=\tau_{i}\dots \tau_1([u])$, we have $||u_i||_A=n$ for $i=1,\dots, k$ and $[u_k]=[u]$ and such that $\phi=\tau_k\dots \tau_1$ in $\Out(F_N)$.
\end{prop}

\begin{rem}\label{rem:R}
By Definition~\ref{d:autgraph}, every oriented edge of $\mathcal T_n$ has a specified Whitehead label, including loop edges. Reading labels gives a homomorphism
\[
\rho_{[u]}:\pi_1(\mathcal T_n[u],[u])\longrightarrow\Out(F_N).
\]
Since at most $2\#\mathcal W_N$ oriented edge germs leave each vertex, the rank of the free group $\pi_1(\mathcal T_n[u],[u])$ is bounded above by a constant
\[
R=R\bigl(N,\#V\mathcal T_n[u]\bigr).
\]
If $[u]$ is $\Out(F_N)$-minimal and $n=\|u\|_A$, then $\mathcal M([u])=V\mathcal T_n[u]$.
\end{rem}

Proposition~\ref{prop:LSST} now directly implies:
\begin{cor}\label{cor:st}
Let  $1\ne u\in F_N$ be such that $[u]$ is $\Out(F_N)$-minimal, and let $n=||u||_A$. Then:
\begin{enumerate}
\item We have $\Stab_{\Out(F_N)}([u])=\rho_{[u]}\left(\pi_1(\mathcal T_n[u], [u])\right)$.
\item We have $\rank\,\Stab_{\Out(F_N)}([u])\le R(N, \#\mathcal M([u]))$, where $R$ is the constant provided by Remark~\ref{rem:R}. 
\end{enumerate}
\end{cor}

\begin{prop}\label{prop:st}
Let $M\ge 1$ be an integer.  Let $M'=M(N,M)$ be the constant provided by Lemma~\ref{lem:M}. Let $u\in Y_N(M)$ (that is $1\ne u\in F_N$ and the orbit $\Out(F_N)[u]$ contains an $M$-minimal element).
Then
\[
\rank\,\Stab_{\Out(F_N)}([u])\le R(N, M'),
\]
where $R$ is the constant provided by Remark~\ref{rem:R}.
\end{prop}
\begin{proof}
Let $[u']$ be an $\Out(F_N)$-minimal element in $\Out(F_N)[u]$, and put $n=\|u'\|_A$. By Proposition~\ref{p:wpr}, $\mathcal M([u])=\mathcal M([u'])=V\mathcal T_n[u']$, and Lemma~\ref{lem:YN} gives $\#\mathcal M([u])\le M'$. Corollary~\ref{cor:st} bounds the rank of $\Stab_{\Out(F_N)}([u'])$. Since the stabilizers of two classes in the same $\Out(F_N)$-orbit are conjugate, the same bound holds for $\Stab_{\Out(F_N)}([u])$.
\end{proof}

\subsection{Algorithmic detectability}\label{s:alg}

For a finite nonempty subset $S\subseteq \mathcal C_N$ denote $||S||_A=\max\{||u||_A|[u]\in S\}$.

\begin{prop}\label{p:algM}
Let $N\ge 2$ and let $M\ge 1$ be fixed integers. Then there exists a linear time, in terms of $|w|_A$, algorithm that, given $w\in F_N$, decides whether or not $[w]$ is $M$-minimal.
\end{prop}
\begin{proof}
Put $B=\#\mathcal W_N$. The number of sequences $\tau_1,\dots, \tau_{M+1}$ of length $M+1$ of elements of $\mathcal W_N$ is equal to $B^{M+1}$. 
Given $w\in F_N$, for each sequence  $\mathbf t=\tau_1,\dots, \tau_{M+1}$ as above we compute the elements $w_1=\tau_1(w)$, $w_2=\tau_2\tau_1(w), \dots, w_{M+1}=\tau_{M+1}\dots \tau_1(w)$ in $F_N$. We then check whether it is true that $||w||_A\ge ||w_1||_A\ge \dots \ge ||w_{M+1}||_A$ and that $[w], [w_1],\dots, [w_{M+1}]\in \mathcal C_N$ are distinct. For each $\mathbf t$ this check can be done in linear time in terms of $|w|_A$ since $M$ is fixed. If we find a sequence $\mathbf t$ of length $M+1$ such that the above condition holds then $[w]$ is not $M$-minimal. Otherwise $[w]$ is $M$-minimal. Since the number of sequences $\mathbf t$ that need to be considered is equal to $B^{M+1}$, which is a constant, the total running time of this algorithm is linear in terms of $|w|_A$. 
\end{proof}

\begin{rem}
Fix an integer $M\ge1$ and rational numbers $0<\epsilon<1$ and $\lambda>1+\epsilon$.

(1) Since $\mathcal W_N$ is finite, one can decide in linear time in $||S||_A$ whether a given set $S\subseteq\mathcal C_N-\{[1]\}$ of cardinality at most $M$ is $(M,\lambda,\epsilon,\mathcal W_N)$-minimizing.

(2) One can also decide whether such an $S$ is $(M,\lambda,\epsilon)$-minimizing. Conditions (1) and (3) of Definition~\ref{d:MLE} are checked directly, and condition (2) is checked with Whitehead's algorithm. To check condition (4), fix $[u]\in S$ and enumerate all conjugacy classes $[v]$ with
\[
||v||_A<\lambda||u||_A.
\]
There are exponentially many such classes in $||u||_A$. For each one, Whitehead's algorithm decides whether $[v]\in\Out(F_N)[u]$; condition (4) holds precisely when every orbit element found in this finite range already belongs to $S$. Thus, for fixed $M,\lambda,\epsilon$, this test has a coarse exponential-time bound in $||S||_A$.

(3) Consequently, one can decide whether a given $[u]$ is $(M,\lambda,\epsilon)$-minimal. Indeed, condition (3) forces every member of a minimizing set containing $[u]$ to have length at most $(1+\epsilon)||u||_A$. Enumerate the finitely many subsets of that ball having cardinality at most $M$ and containing $[u]$, and apply the test in (2). Since $M$ is fixed, this again gives an exponential-time procedure in $||u||_A$.
\end{rem}

It turns out that deciding whether an element $[u]$ is $(M,\lambda,\epsilon, \mathcal W_N)$-minimal can be done in linear time in $||u||_A$ (under slightly more stringent assumptions in $\lambda,\epsilon$).

\begin{lem}\label{lem:aux1}
Let $M\ge 1$ be an integer, let $\lambda>1$, $0<\epsilon<1$ be such that $\epsilon<\lambda-1$ and $\lambda \frac{1-\epsilon}{1+\epsilon}>1$. Let $S\subseteq \mathcal C_N$ be a finite set of conjugacy classes of nontrivial elements of $F_N$ such that $S$ is $(M,\lambda,\epsilon, \mathcal W_N)$-minimizing. Let $[u]\in S$.

Then for $[u']\in \mathcal C_N, [u']\ne [u]$ the following conditions are equivalent:

\begin{enumerate}
\item We have $[u']\in S$.
\item There exists a chain $\tau_1,\dots, \tau_k\in \mathcal W_N$ such that $k\le M$, that $\tau_k\dots\tau_1[u]=[u']$ and that with $[u_0]=[u]$ $[u_i]=\tau_i\dots\tau_1[u]$ we have $||u_{i+1}||_A\le (1+\epsilon)||u_i||_A$ for $0\le i<k$.
\end{enumerate}

\end{lem}
\begin{proof}
Assume first that (2) holds. Starting with $[u_0]=[u]\in S$, Lemma~\ref{lem:aux}(1) applies successively to the displayed inequalities and shows that every $[u_i]$ lies in $S$. Hence $[u']=[u_k]\in S$, proving (1).

Conversely, suppose that $[u']\in S$. Then
\[
\|u'\|_A\le(1+\epsilon)\|u\|_A.
\tag{3.1}
\]
Choose $\phi\in\Out(F_N)$ with $\phi([u])=[u']$. If $\|u'\|_A\le\|u\|_A$, apply Proposition~\ref{p:wprs} directly to $\phi$. If $\|u'\|_A>\|u\|_A$, apply that proposition to $\phi^{-1}$ and the pair $[u'],[u]$, and then reverse the resulting chain, replacing every Whitehead move by its inverse. In either case we obtain Whitehead moves $\tau_1,\dots,\tau_k$ such that, with
\[
[u_0]=[u],\qquad [u_i]=\tau_i\cdots\tau_1([u]),
\]
one has $[u_k]=[u']$ and
\[
\|u_i\|_A\le\max\{\|u\|_A,\|u'\|_A\}
\le(1+\epsilon)\|u\|_A
\qquad(0\le i\le k).
\tag{3.2}
\]
Deleting closed subchains if necessary, we may assume that the classes $[u_0],\dots,[u_k]$ are distinct.

We claim that every step satisfies
\[
\|u_i\|_A\le(1+\epsilon)\|u_{i-1}\|_A.
\tag{3.3}
\]
Otherwise, let $i$ be the first index for which (3.3) fails. Lemma~\ref{lem:aux}(1) shows inductively that $[u_0],\dots,[u_{i-1}]$ lie in $S$. Condition (3) of Definition~\ref{d:MLEW} then implies that $[u_i]\notin S$, and condition (4)(ii) gives
\[
\|u_i\|_A\ge\lambda\|u_{i-1}\|_A
\ge\lambda(1-\epsilon)\|u\|_A
>(1+\epsilon)\|u\|_A,
\]
where the final inequality is the hypothesis
$\lambda(1-\epsilon)/(1+\epsilon)>1$. This contradicts (3.2), proving the claim.

Another induction using Lemma~\ref{lem:aux}(1) now shows that all $[u_i]$ lie in $S$. They are distinct and $\#S\le M$, so $k\le M-1$, in particular $k\le M$. Thus (2) holds.
\end{proof}

\begin{cor}\label{c:alg}
Let $M\ge 1$ be an integer, let $\lambda>1$, $0<\epsilon<1$ be rational numbers such that $\epsilon<\lambda-1$  and $\lambda \frac{1-\epsilon}{1+\epsilon}>1$.
Then there is an algorithm that, given $1\ne u\in F_N$ decides in linear time in $|u|_A$ whether or not $[u]$ is $(M,\lambda,\epsilon, \mathcal W_N)$-minimal.
\end{cor}
\begin{proof}
Suppose we are given an input $1\ne u\in F_N$. We need to decide if there exists an $(M,\lambda,\epsilon, \mathcal W_N)$-minimizing set $S$ containing $[u]$.

We first enumerate all chains of $k\le M$ Whitehead moves as in part (2) of Lemma~\ref{lem:aux1} and collect all $[u']$ reachable from $[u]$ by applying such chains. Denote the resulting subset of $\mathcal C_N$ by $S'$. Computing $S'$ from $[u]$ takes at most linear time in $|u|_A$ since $M$ is fixed and the set $\mathcal W_N$ is also finite and fixed.

Lemma~\ref{lem:aux1} implies that if $[u]$ belongs to some $(M,\lambda,\epsilon, \mathcal W_N)$-minimizing set $S$ then $S=S'$.
We then check if conditions (1)-(4) of  Definition~\ref{d:MLEW} hold for $S'$. Again this can be done in linear time in $|u|_A$ since $M$ is fixed.

We conclude that $[u]$ is $(M,\lambda,\epsilon, \mathcal W_N)$-minimal if and only if conditions (1)-(4) of  Definition~\ref{d:MLEW} do hold for $S'$.
\end{proof}

\section{Geodesic currents on free groups}

We provide some basic background on geodesic currents on $F_N$ here and refer the reader to \cite{Ka06,KL09,KL10} for further details. For the remainder of this section, let $F_N$ be a free group of finite rank $N\ge 2$. We denote by $\partial F_N$ the hyperbolic boundary of $F_N$ and denote $\dd:=\{(x,y)|x,y\in\partial F_N, x\ne y\}$.
We give $\dd$ the subspace topology from $\partial F_N\times \partial F_N$ and endow $\partial^2 F_N$ with the natural diagonal translation action of $F_N$ by homeomorphisms. The space $\partial^2 F_N$ also comes with a natural ``flip'' involution $\varpi:\partial^2F_N\to\partial^2 F_N$, $\varpi:(x,y)\mapsto (y,x)$. The boundary $\partial F_N$ is homeomorphic to the Cantor set, and $\partial^2 F_N$ is a locally compact totally disconnected but non-compact metrizable topological space.

\subsection{Basic notions}

\begin{defn}[Oriented currents and symmetrization]
An \emph{oriented geodesic current} on $F_N$ is a locally finite positive Borel measure on $\partial^2F_N$ that is invariant under the diagonal action of $F_N$; flip invariance is not required. We denote the space of oriented currents by $\Curr^+(F_N)$ and equip it with the weak-* topology. If $\alpha$ is an oriented current, define
\[
\operatorname{Sym}(\alpha):=\alpha+\varpi_*\alpha.
\]
Then $\operatorname{Sym}(\alpha)$ is flip-invariant.
\end{defn}

\begin{defn}
A \emph{geodesic current} on $F_N$ is a locally finite (i.e. finite on compact subsets) positive Borel measure $\nu$ on $\partial^2 F_N$ such that $\nu$ is $F_N$-invariant and flip-invariant. The set of all geodesic currents on $F_N$ is denoted $\CN$.
\end{defn}
The set $\CN$ is equipped with the weak-* topology, which makes $\CN$ locally compact. Any automorphism $\Phi\in\Aut(F_N)$ is a quasi-isometry of $F_N$ and hence extends to a homeomorphism, which we still denote by $\Phi:\partial F_N\to\partial F_N$. Diagonally extending this homeomorphism we also get a homeomorphism $\Phi:\partial^2 F_N\to\partial^2 F_N$. There is a natural left action of $\Aut(F_N)$ by homeomorphisms on $\CN$, where for $\Phi\in\Aut(F_N)$ and $\nu\in\CN$ we have $(\Phi\nu)(S)=\nu(\Phi^{-1}(S))$ for $S\subseteq \dd$. The subgroup ${\rm Inn}(F_N)\le \Aut(F_N)$ is contained in the kernel of this action, and therefore the action descends to the action of $\Out(F_N)$ on $\CN$. There is also a multiplication by a scalar action of $\R_{>0}$ on $\CN-\{0\}$, with the quotient space $\PCN=(\CN-\{0\})/\R_{>0}$, equipped with the quotient topology. The space $\PCN$ is compact, although infinite dimensional. For $0\ne \nu\in\CN$ we denote the $\R_{>0}$-equivalence class of $\nu$ by $[\nu]$. Thus $[\nu]=\{c\nu| c\in \R_{>0}\}$ and $[\nu]\in\PCN$. We call elements of $\PCN$ \emph{projectivized geodesic currents} on $F_N$.

Let $1\ne g\in F_N$. Then $g$ determines a pair of distinct ``poles'' $g^{-\infty}, g^{\infty}\in \partial F_N$, where $g^\infty=\lim_{n\to\infty} g^n$ and $g^{-\infty}=\lim_{n\to\infty} g^{-n}$ in $F_N\cup\partial F_N$. Thus  $(g^{-\infty},g^\infty)\in\dd$. For $h\in F_N$ we have $hg^\infty=(hgh^{-1})^\infty$, and we also have $g^{-\infty}=(g^{-1})^\infty$.

\begin{defn}[Counting and rational currents]
Let $1\ne g\in F_N$. The \emph{oriented counting current} of $g$ is
\[
\eta_g^+:=\sum_{h\in F_N/\langle g\rangle}
\delta_{h(g^{-\infty},g^\infty)}.
\]
The usual counting current is its symmetrization:
\[
\eta_g:=\operatorname{Sym}(\eta_g^+)
=\sum_{h\in F_N/\langle g\rangle}
\left(\delta_{h(g^{-\infty},g^\infty)}+
\delta_{h(g^\infty,g^{-\infty})}\right).
\]
We call currents of the form $c\eta_g$, where $c>0$, \emph{rational currents}.
\end{defn}
It is known that the set of all rational currents is a dense subset of $\CN$, and that for any $1\ne g\in F_N$ and any $u\in F_N$ we have $\eta_g=\eta_{ugu^{-1}}=\eta_{g^{-1}}$. Therefore we also denote $\eta_{[g]}:=\eta_g$ where $[g]$ is the conjugacy class of $g$ in $F_N$. Moreover, for $\phi\in\Aut(F_N)$ and $1\ne g\in F_N$, one has $\phi\eta_g=\eta_{\phi(g)}$.

For a fixed free basis $A$, the standard suspension construction gives a linear homeomorphism
\[
\mathcal I_A:\Curr^+(F_N)\longrightarrow
\mathcal M_{\mathrm{flow}}\bigl(T^1(T_A/F_N)\bigr)
\]
from oriented currents to geodesic-flow-invariant Radon measures on the oriented geodesic-flow space. It is characterized on oriented counting currents by
\[
\mathcal I_A(\eta_g^+)=||g||_A D_g,
\]
where $D_g$ is normalized arclength measure on the oriented closed geodesic determined by $g$. In particular,
\[
\mathcal I_A\left(\frac{\eta_g^+}{||g||_A}\right)=D_g.
\tag{4.1}
\]
This correspondence is also obtained by lifting a flow-invariant measure to $\partial^2F_N\times\R$ and disintegrating it as an oriented current times Lebesgue measure along the flow lines.

\subsection{Simplicial charts and weights}

We adopt the conventions of~\cite{DKT15} regarding graphs. All graphs are 1-cell complexes, where 0-cells are called vertices and 1-cells are called topological edges. Every topological edge is homeomorphic to an interval $(0,1)$ and thus admits exactly two orientations. An \emph{oriented edge} of a graph is a topological edge with a choice of an orientation. The same topological edge with the opposite orientation is denoted $e^{-1}$. The set of all oriented edges of a graph $\Delta$ is denoted $E\Delta$. We also denote by $V\Delta$ the set of all vertices of $\Delta$. Unless specified otherwise, by an edge of a graph we always mean an oriented edge. Every oriented edge $e\in E\Delta$ has an \emph{initial vertex} denoted $o(e)\in V\Delta$ and a \emph{terminal vertex} $t(e)\in V\Delta$. We also have $o(e^{-1})=t(e)$ and $t(e^{-1})=o(e)$.  An \emph{edge-path} $\gamma$ of length $n\ge 1$ in $\Delta$  is a sequence of edges $e_1,\dots, e_n$ such that $t(e_i)=o(e_{i+1})$. We also consider a vertex $v$ of $\Delta$ to be a path of length $0$. An edge-path $\gamma$ in $\Delta$ is \emph{reduced} if it does not contain subpaths of the form $e,e^{-1}$ where $e\in E\Delta$. We denote by $|\gamma|$ the length of an edge-path $\gamma$.

\begin{defn}[Simplicial chart]
Let $F_N$ be free of rank $N\ge 2$. A \emph{simplicial chart} on $F_N$ is a pair $(\Gamma, \kappa)$ where $\Gamma$ is a finite connected oriented graph with all vertices of degree $\ge 3$ and with the first Betti number $b(\Gamma)=N$, and where $\kappa: F_N\to \pi_1(\Gamma,x_0)$ is a group isomorphism (with $x_0\in V\Gamma$ some base vertex), called a \emph{marking}.
\end{defn}

When talking about simplicial charts, we usually suppress explicit mention of $\kappa$.
We equip $\Gamma$ and $T_0=\widetilde {(\Gamma,x_0)}$ with simplicial metrics, where every edge has length $1$.
In this setting we denote by $\Omega(\Gamma)$ the set of all semi-infinite reduced edge-paths $e_1,e_2,\dots, $ in $\Gamma$. For $n\ge 1$ denote by $\Omega_n(\Gamma)$ the set of all  reduced edge-paths $e_1,e_2,\dots, e_n$ of length $n$ in $\Gamma$. Also denote $\Omega_\ast=\cup_{n=1}^\infty \Omega_n(\Gamma)$.

If $A=\{a_1,\dots, a_N\}$ is a free basis of $F_N$, then the graph $R_A$, with a single vertex $x_0$ and with $N$ petal-edges marked $a_1,\dots, a_N$, is a simplicial chart on $F_N$. In this case the corresponding covering tree $T_A:=\widetilde R_A$ is exactly the Cayley tree of $F_N$ with respect to $A$. We refer to such simplicial chart $R_A$ as an \emph{$N$-rose}.

For a simplicial chart $\Gamma$, the marking $\kappa$ induces an $F_N$-equivariant quasi-isometry $F_N\to T_0$, which we use to identify $\partial F_N$ with $\partial T_0$. For $(x,y)\in \partial^2F_N$ denote by $\gamma_{x,y}$ the bi-infinite geodesic in $T_0$ from $x$ to $y$.  The group $F_N=\pi_1(\Gamma,x_0)$ acts on $T_0=\widetilde \Gamma$ by covering transformations, which is a free and isometric discrete action with $T_0/F_N=\Gamma$.

\begin{defn}[Cylinders and weights]
Let $\Gamma$ be a simplicial chart on $F_N$, with $T_0=\widetilde \Gamma$.

(1) For two distinct vertices $p,q\in T_0$ denote by $Cyl_\Gamma([p,q])$ the set of all $(x,y)\in\partial^2 F_N$ such that the geodesic $\gamma_{x,y}$, oriented from $x$ to $y$, traverses the segment $[p,q]$ from $p$ to $q$. The set $Cyl_\Gamma([p,q])\subseteq \dd$ is called the \emph{cylinder set} corresponding to $[p,q]$.

For any $g\in F_N$ and any $p,q\in VT_0, p\ne q$ we have $gCyl_\Gamma([p,q])=Cyl_\Gamma([gp,gq])$. The cylinder sets $Cyl_\Gamma([p,q])\subseteq \dd$ are compact and open, and the collection of all such cylinder sets forms a basis for the subspace topology on $\dd$ defined above.

(2) For a geodesic current $\eta\in\CN$ denote by $\langle v,\eta\rangle_\Gamma:=\eta\left( Cyl_\Gamma([p,q])\right)$ where $[p,q]$ is any lift of $v$ to $T_0$.  The number $0\le \langle v,\eta\rangle_\Gamma<\infty$ is called the \emph{weight} of $v$ in $\eta$ with respect to $\Gamma$.
\end{defn}

If $\Gamma=R_A$ is an $N$-rose, we use the subscript $A$ rather than $R_A$ for chart-related notations. E.g. $\langle v,\eta\rangle_A:=\langle v,\eta\rangle_{R_A}$, etc.

\begin{prop}\cite{Ka06}
Let $F_N$ be free of rank $N\ge 2$ and let $\Gamma$ be a simplicial chart on $F_N$. Then:

\begin{enumerate}
\item For $\eta, \eta_n\in \CN$, where $n=1,2,\dots$, we have $\lim_{n\to\infty} \eta_n=\eta$ in $\CN$ if and only if for every $v\in\Omega_\ast(\Gamma)$ we have
\[
\lim_{n\to\infty} \langle v,\eta_n\rangle_\Gamma=\langle v, \eta\rangle_\Gamma.
\]
\item Let $\eta\in \CN$. Then for every $k\ge 1$ and every $v\in\Omega_k(\Gamma)$ we have
\[
\langle v, \eta\rangle_\Gamma=\sum_{e\in E\Gamma \text{ with } ve\in \Omega_{k+1}(\Gamma)}\langle ve, \eta\rangle_\Gamma=\sum_{e'\in E\Gamma \text{ with } e'v\in \Omega_{k+1}(\Gamma)}\langle e'v, \eta\rangle_\Gamma. \tag{$\ddag$}
\]

Moreover, any system of finite nonnegative weights $a(v)$, $v\in\Omega_\ast(\Gamma)$, satisfying the two switch conditions
\[
a(v)=\sum_e a(ve)=\sum_{e'}a(e'v)
\]
(with the sums taken over reduced extensions) and the symmetry relations $a(v)=a(v^{-1})$ uniquely determines a current $\eta\in\CN$ with $\langle v,\eta\rangle_\Gamma=a(v)$.
\end{enumerate}
\end{prop}

Condition $(\ddag)$ is often called the \emph{switch condition} for $\Gamma$.

For $v\in \Omega_\ast(\Gamma)$ and a nondegenerate closed reduced and cyclically reduced edge-path $w$ in $\Gamma$, denote by $\langle v,w\rangle_\Gamma$ the number of ways in which $v$ can be read, reading forwards or backwards, in a circle of length $|w|$ labelled by $w$. The number $\langle v,w\rangle_\Gamma\ge 0$ is called the \emph{number of occurrences} of $v$ in $w$.
A key useful fact that follows from the definitions is:

\begin{lem}
Let $F_N$ be free of rank $N\ge 2$ and let $\Gamma$ be a simplicial chart on $F_N$. Let $v\in \Omega_\ast(\Gamma)$ and let $w$ be a nondegenerate closed reduced and cyclically reduced edge-path in $\Gamma$. Then $\langle v,w\rangle_\Gamma=\langle v,\eta_w\rangle_\Gamma$.
\end{lem}
\hfill $\qed$

\begin{defn}[Uniform current]\label{d:uc}
Let $F_N=F(A)$ be free of rank $N\ge 2$ with a free basis $A$. The \emph{uniform current} $\nu_A\in\CN$ corresponding to $A$ is the current given by the weights $\langle v,\nu_A\rangle_A=\frac{1}{N(2N-1)^{k-1}}$ for every $1\ne v\in F_N$ with $|v|_A=k\ge 1$.
\end{defn}

For a current $\eta\in\CN$ the \emph{support} $\supp(\eta)\subseteq \dd$ is
\[
\supp(\eta):=\dd\setminus \bigcup \{U\subseteq \dd| U \text{ is open and } \eta(U)=0\}.
\]
Thus $\supp(\eta)$ is a closed $F_N$-invariant subset of $\dd$.

\begin{rem}\label{r:sup}
Let $\Gamma$ be a simplicial chart on $F_N$. If $\eta\in\CN$ and $(x,y)\in\dd$ then $(x,y)\in\supp(\eta)$ if and only if every finite nondegenerate edge subpath of $\gamma_{x,y}$ projects to a reduced edge-path $v$ in $\Gamma$ with $\langle v,\eta\rangle_\Gamma>0$.
\end{rem}

\subsection{Geometric intersection form}

We refer the reader to \cite{B,FM,KL10,V15} for the background and basic info regarding the Outer space, and only recall a few facts and definitions here.
Denote by $\cvn$ the (unprojectivized) Culler-Vogtmann Outer space for $F_N$. Elements of $\cvn$ are equivariant $F_N$-isometry classes of free and discrete minimal isometric actions of $F_N$ on $\R$-trees. In particular, if $\Gamma$ is a simplicial chart on $F_N$ then $T_0=\widetilde \Gamma$ defines a point of $\cvn$.
There is a natural ``axes'' topology on $\cvn$ and a (right) action of $\Out(F_N)$ on $\cvn$ by homeomorphisms.  Moreover, the closure $\cvnbar$ of $\cvn$ in the axes topology is known to consist of all minimal nontrivial ``very small'' isometric actions on $F_N$ on $\R$-trees (again considered up to $F_N$-equivariant isometry), and the action of $\Out(F_N)$ extends to $\cvnbar$. For $T\in\cvnbar$ and $g\in F_N$ denote by $||g||_T$ the \emph{translation length} of $g$ in $T$, that is $||g||_T=\inf_{x\in T} d_T(x,gx)$.

A key result of Kapovich and Lustig~\cite{KL09} is:

\begin{prop}\label{p:int}
Let $F_N$ be free of finite rank $N\ge 2$. Then there exists a continuous \emph{geometric intersection form}
\[
\langle -\, , \, - \rangle:\cvnbar\times \CN\to \R_{\ge 0}
\]
satisfying the following properties:
\begin{enumerate}
\item The map $\langle -\, , \, - \rangle$ is $\R_{\ge 0}$-homogeneous with respect to the first argument and $\R_{\ge 0}$-linear with respect to the second argument.
\item For every $\phi\in\Out(F_N)$, every $T\in\cvnbar$ and every $\eta\in\CN$ we have
\[
\langle T,\phi\eta\rangle=\langle T\phi,\eta\rangle.
\]
\item For every $1\ne g\in F_N$ and every $T\in\cvnbar$ we have $\langle T,\eta_g\rangle=||g||_T$.
\end{enumerate}
\end{prop}

In view of the above proposition, for $T\in\cvnbar$ and $\eta\in\CN$ we denote $||\eta||_T=\langle T,\eta\rangle$. If $A$ is a free basis of $F_N$ and $T_A$ is the Cayley tree of $F_N$ with respect to the free basis $A$, then for $\eta\in \CN$ we denote $||\eta||_A:=||\eta||_{T_A}=\langle T_A,\eta\rangle$.

For every $T\in\cvnbar$ there is an associated \emph{dual lamination} $L(T)\subseteq \dd$, which is a certain closed $F_N$-invariant and flip-invariant subset of $\dd$ recording the information about sequences of elements of $F_N$ with translation length in $T$ converging to $0$. We refer the reader to \cite{KL10} for the precise definition of $L(T)$ and additional details.

We need the following key result of \cite{KL10}:

\begin{prop}\label{p:KL100}
Let $T\in\cvnbar$ and $\eta\in\CN$. Then $||\eta||_T=0$ if and only if $\supp(\eta)\subseteq L(T)$.
\end{prop}

\section{Filling geodesic currents}

Whenever edge-paths occur in this section, $\Gamma$ denotes a fixed simplicial chart on $F_N$.

\begin{defn} Let $F_N$ be free of rank $N\ge 2$.
\begin{enumerate}
\item An element $g\in F_N$ is \emph{filling} in $F_N$ if for every $T\in\cvnbar$ we have $||g||_T>0$.
\item A current $\eta\in\CN$ is \emph{filling} in $F_N$ if for every $T\in\cvnbar$ we have $||\eta||_T>0$.
\end{enumerate}
\end{defn}
Thus an element $1\ne g\in F_N$ is filling if and only if $\eta_g$ is a filling.

One of the main results of \cite{KL10} is:

\begin{prop}\cite[Corollary~1.6]{KL10}\label{p:fs}
Let $\eta\in\CN$ be such that $\supp(\eta)=\dd$. Then $\eta$ is filling in $F_N$.
\end{prop}

We will sometimes say that a current $\eta\in \CN$ \emph{has full support} if $\supp(\eta)=\dd$.

Remark~\ref{r:sup} directly implies:

\begin{prop}
Let $\Gamma$ be a simplicial chart on $F_N$.

Then $\eta\in\CN$ has full support if and only if for every nondegenerate edge-path $v$ in $\Gamma$ we have $\langle v,\eta\rangle_\Gamma>0$.
\end{prop}


\begin{lem}\label{lem:w}
Let $0\ne \nu\in\CN$.
Let $w$ be a nondegenerate closed reduced and cyclically reduced edge-path in $\Gamma$ such that for every $n\ge 1$ we have $\langle w^n,\nu\rangle_\Gamma>0$. Then  $\supp(\eta_w)\subseteq \supp(\nu)$.
\end{lem}
\begin{proof}
Let $p_+=(w^{-\infty},w^{\infty})$. The cylinders determined by longer and longer finite subsegments of the axis of $w$ form a neighborhood basis at $p_+$. Every such segment is, after translating along the axis and possibly passing to a subpath, contained in a lift of $w^n$ for some $n$. Since $\langle w^n,\nu\rangle_\Gamma>0$, the cylinder of that lift of $w^n$ has positive $\nu$-measure and is contained in the cylinder of the shorter segment. Hence every cylinder neighborhood required by the support criterion has positive measure. Thus $p_+\in\supp(\nu)$. Flip invariance gives $p_-=(w^{\infty},w^{-\infty})\in\supp(\nu)$. Finally,
\[
\supp(\eta_w)=\bigcup_{h\in F_N}h\{p_+,p_-\},
\]
and $F_N$-invariance of $\supp(\nu)$ gives $\supp(\eta_w)\subseteq\supp(\nu)$.
\end{proof}

\begin{lem}\label{lem:f}
Let $1\ne g\in F_N$ be a filling element and let $0\ne \nu\in \CN$ be a current such that $\supp(\eta_g)\subseteq \supp(\nu)$. Then $\nu$ is a filling current.
\end{lem}

\begin{proof}
Suppose, on the contrary, that $\nu$ is not filling. Then there exists $T\in \cvnbar$ such that $\langle T,\nu\rangle=0$.
By \cite[Theorem~1.1]{KL10} this implies that $\supp(\nu)\subseteq L(T)$.  Hence $\supp(\eta_g)\subseteq L(T)$ as well. Therefore, again by \cite[Theorem~1.1]{KL10}, we have $0=\langle T,\eta_g\rangle=||g||_T$, which contradicts that $g$ is filling.
\end{proof}




\begin{cor}\label{c:z}
Let $z$ be a nondegenerate closed reduced and cyclically reduced edge-path in $\Gamma$ representing the conjugacy class of a filling element $g\in F_N$.

Let $0\ne \nu\in\CN$ be such that for every $n\ge 1$ we have $\langle z^n,\nu\rangle_\Gamma>0$. Then $\nu$ is a filling current.
\end{cor}
\begin{proof}
Lemma~\ref{lem:w} implies that $\supp(\eta_g)\subseteq \supp(\nu)$.
Therefore, by Lemma~\ref{lem:f}, the current $\nu$ is filling.
\end{proof}

\begin{prop}\label{p:ai}
Let $0\ne\nu\in\CN$ be such that for some free basis $A=\{a_1,\dots, a_N\}$ the following holds. For $i=1,\dots, N$ let $w_i$ be a closed reduced and cyclically reduced edge-path in $\Gamma$ representing the conjugacy class of $a_i$ in $F_N$. For $1\le i<j\le N$ let $w_{i,j}$ be a closed reduced and cyclically reduced edge-path in $\Gamma$ representing the conjugacy class of $a_ia_j$ in $F_N$. Suppose that $\langle w_i^n,\nu\rangle_\Gamma>0$ for every $i=1,\dots,N$ and every $n\ge1$, and that $\langle w_{i,j}^n,\nu\rangle_\Gamma>0$ for every $1\le i<j\le N$ and every $n\ge1$. Then  the current $\nu\in\CN$ is filling.
\end{prop}
\begin{proof}
Indeed, suppose $\nu$ is not filling. Then there exists $T\in\cvnbar$ such that $\langle T,\nu\rangle=0$. By \cite[Theorem~1.1]{KL10} this implies that $\supp(\nu)\subseteq L(T)$.

Lemma~\ref{lem:w} implies that for all $i=1,\dots, N$ we have  $\supp(\eta_{a_i})\subseteq \supp(\nu)$,  and for all $1\le i<j\le N$ we have $\supp(\eta_{a_ia_j})\subseteq \supp(\nu)$.
Since $\supp(\nu)\subseteq L(T)$, \cite[Theorem~1.1]{KL10}  implies that for $i=1,\dots, N$
\[
0=\langle T, \eta_{a_i}\rangle=||a_i||_T
\]
and for all $1\le i<j\le N$ we have
\[
0=\langle T, \eta_{a_ia_j}\rangle=||a_ia_j||_T
\]
Thus all $a_i$ and $a_ia_j$ act elliptically on $T$ and so have nonempty fixed sets in $T$.

For $1\le i<j\le N$, the elements $a_i,a_j,a_i a_j$ act elliptically on $T$, and therefore, by \cite[Proposition~1.8]{Pa}, $\Fix_T(a_i)\cap \Fix_T(a_j)\ne\varnothing$. Thus $\Fix_T(a_1),\dots \Fix_T(a_N)$ are nonempty subtrees of $T$ with pairwise nonempty intersections. Therefore $\cap_{i=1}^N \Fix_T(a_i)\ne\varnothing$. Hence $F_N$ has a global fixed point in $T$, which contradicts the fact that $T\in\cvnbar$ is a nontrivial $F_N$-tree.
\end{proof}

\begin{prop}
Let $F_N=F(A)$ (where $N\ge 2$) and let $w\in F(A)$ be a freely and cyclically reduced word such that for every $v\in F(A)$ with $|v|_A=3$, the word $v$ occurs as a subword of some cyclic permutation of $w$ or of $w^{-1}$. Then:
\begin{enumerate}
\item The element $w\in F_N$ is filling.
\item If $0\ne \nu\in \CN$ is such that for all $n\ge 1$ $\langle w^n,\nu\rangle_A>0$, then the current $\nu$ is filling in $F_N$.
\end{enumerate}
\end{prop}
\begin{proof}
Part (1) is exactly \cite[Corollary~5.6]{CM15}.

Now part (1) implies part (2) by Corollary~\ref{c:z}.
\end{proof}

\section{\texorpdfstring{Filling currents and $M$-minimality}{Filling currents and M-minimality}}

Also, as before, we denote by $T_A$ the Cayley tree of $F_N$ with respect to the free basis $A$. Thus $T_A$ is a simplicial tree with all edges of length $1$.

\begin{defn}
Let $0\ne \nu\in \CN$. The \emph{automorphic distortion spectrum} of $\nu$ with respect to the free basis $A$ of $F_N$ is the set
\[
D_A(\nu):=\{\|\phi\nu\|_A: \phi\in\Out(F_N)\}.
\]
Also denote $J_A(\nu):=\inf D_A(\nu)$.
\end{defn}

\begin{rem}
Thus $D_A(\nu)\subseteq\R_{>0}$. For every $\phi\in\Out(F_N)$,
\[
||\phi\nu||_A=\langle T_A,\phi\nu\rangle
=\langle T_A\phi,\nu\rangle.
\]
If $B$ is another free basis, choose $\psi\in\Out(F_N)$ with $T_B=T_A\psi$. Then
\[
||\phi\nu||_B=\langle T_A\psi,\phi\nu\rangle
=\langle T_A,\psi\phi\nu\rangle.
\]
As $\phi$ ranges over $\Out(F_N)$, so does $\psi\phi$; hence $D_B(\nu)=D_A(\nu)$. We retain the subscript because the minimizing automorphisms and the associated length function are expressed relative to the chosen basis.

Note also that for $1\ne w\in F_N$ and $\phi\in \Out(F_N)$ we have $\langle T_A, \phi \eta_w \rangle=\langle T_A, \eta_\phi(w)\rangle=||\phi(w)||_A$. Therefore in this case
$D_A(\eta_w)=\{||\phi(w)||_A: \phi\in \Out(F_N)\}\subseteq \Z_{>0}$, and $J_A(\eta_w)$ is the smallest $||.||_A$-length of elements in the orbit $\Out(F_N)[w]$.
\end{rem}

We need the following useful result essentially proved in \cite[Theorem~11.2]{KL10}:
\begin{prop}\label{p:KL10}
Let $\nu\in \CN$ (where $N\ge 2$) be a filling current and let $A$ be a free basis of $F_N$.  Then:
\begin{enumerate}
\item The set $D_A(\nu)$ is a discrete unbounded subset of $[0,\infty)$.

\item For every $C>0$ the set $\{\phi\in \Out(F_N): ||\phi \nu||_A\le C\}$ is finite.
\end{enumerate}
\end{prop}
\begin{proof}
Theorem~11.2 of \cite{KL10} is stated for currents of full support. Its proof uses that hypothesis only to obtain
\[
\langle T_\infty,\nu\rangle>0
\]
for a limiting tree $T_\infty\in\cvnbar$. The definition of a filling current gives exactly this positivity for every $T_\infty\in\cvnbar$. The remainder of the proof of \cite[Theorem~11.2]{KL10} therefore applies verbatim and gives the finite-sublevel conclusion in (2). It follows that $D_A(\nu)\cap[0,C]$ is finite for every $C>0$, and hence that $D_A(\nu)$ is discrete. If $D_A(\nu)$ were bounded, then all of the infinite group $\Out(F_N)$ would lie in a single finite sublevel set, a contradiction. Thus $D_A(\nu)$ is unbounded, proving (1).
\end{proof}

Proposition~\ref{p:KL10} immediately implies:

\begin{cor}
For $F_N$ and $A$ as in Proposition~\ref{p:KL10} let $\nu\in \CN$ be a filling current. Then:
\begin{enumerate}
\item We have $J_A(\nu)\in D_A(\nu)$, so that $J_A(\nu)=\min D_A(\nu)$.
\item The set $\mathfrak W_{A,\nu}=\{\phi\in\Out(F_N): ||\phi\nu||_A=J_A(\nu)\}$ is finite and nonempty.
\end{enumerate}
\end{cor}

We also record the following useful general consequence of Proposition~\ref{p:KL10}:

\begin{prop}\label{p:st}
Let $\nu\in \CN$ (where $N\ge 2$) be a filling current. Then:

\begin{enumerate}
\item The stabilizer $\Stab_{\Out(F_N)}(\nu)$ is finite.
\item We have $\Stab_{\Out(F_N)}(\nu)=\Stab_{\Out(F_N)}([\nu])$, where $[\nu]\in \PCN$ is the projective class of $\nu$. 
\item If $\tau\in \mathcal W_N$ is a Whitehead move of the second kind such that $\tau\ne 1$ in $\Out(F_N)$ then $\tau\nu\ne \nu$ and $\tau[\nu]\ne [\nu]$. 
\end{enumerate}
\end{prop} 
\begin{proof}
Let $A$ be a free basis of $F_N$. Put $C=||\nu||_A$.  For every $\phi\in \Stab_{\Out(F_N)}(\nu)$ we have $||\phi\nu||_A=||\nu||_A=C$. Therefore
\[
\Stab_{\Out(F_N)}(\nu)\subseteq \{\phi\in \Out(F_N): ||\phi \nu||_A\le C\}
\]
and hence $\Stab_{\Out(F_N)}(\nu)$ is finite by Proposition~\ref{p:KL10}.  Thus (1) is verified.

For (2), suppose that $\psi\in \Stab_{\Out(F_N)}([\nu])$. Thus $\psi\nu=t\nu$ for some $t>0$. 

Suppose first that $t\ne 1$. Then $\psi^{-1}\nu=\frac{1}{t}\nu$. Therefore, after possibly replacing $\psi$ by $\psi^{-1}$, we have $\psi\nu=t\nu$ with $t<1$. Hence for every $n\ge 1$ $\psi^n \nu=t^n\nu$. Since $t\ne 1$, this implies that $\psi^n\nu\ne \nu$, so that $\psi$ has infinite order in $\Out(F_N)$. Moreover, since $t<1$, for every $n\ge 1$ we have $||\psi^n\nu||_A=t^n||\nu||_A\le ||\nu||_A=C$. Hence
\[
\{\psi^n| n\ge 1\}\subseteq   \{\phi\in \Out(F_N): ||\phi \nu||_A\le C\},
\]
which contradicts the fact that by Proposition~\ref{p:KL10} the latter set is finite. Thus $t=1$, so that $\psi\nu=t\nu=\nu$ and $\psi\in \Stab_{\Out(F_N)}(\nu)$. Since $\psi\in \Stab_{\Out(F_N)}([\nu])$ was arbitrary, it follows that $\Stab_{\Out(F_N)}([\nu])=\Stab_{\Out(F_N)}(\nu)$. The inclusion $\Stab_{\Out(F_N)}(\nu)\subseteq \Stab_{\Out(F_N)}([\nu])$ is obvious. Hence $\Stab_{\Out(F_N)}(\nu)=\Stab_{\Out(F_N)}([\nu])$, and (2) holds, as required. 

It remains to justify the infinite-order assertion used in (3). Represent $\tau$ by a Whitehead automorphism $\Phi$ of the second kind with multiplier $a\in A^{\pm1}$, and let $a_0\in A$ be the underlying basis element, so that $a=a_0^{\pm1}$. If the induced map on $H_1(F_N;\Z)$ is nontrivial, its matrix has the form $I+Q$, where $Q\ne0$, the image of $Q$ is contained in the line spanned by $a$, and $Q(a)=0$. Hence $Q^2=0$ and
\[
(I+Q)^n=I+nQ\ne I \qquad(n\ge1),
\]
so the outer class of $\Phi$ has infinite order.

Suppose instead that $\Phi$ acts trivially on $H_1(F_N;\Z)$. Then every $x\in A\setminus\{a_0\}$ is mapped either to itself or to $a^{-1}xa$; thus $\Phi$ is a partial conjugation. Conjugating no element of $A\setminus\{a_0\}$ gives the identity, while conjugating all of them gives an inner automorphism. Since $\tau\ne1$ in $\Out(F_N)$, $\Phi$ conjugates a nonempty proper subset of $A\setminus\{a_0\}$. Choose $x$ in that subset and $y$ in its complement. Then
\[
\Phi^n(x)=a^{-n}xa^n,\qquad \Phi^n(a)=a,\qquad \Phi^n(y)=y.
\]
If $\Phi^n$ were inner, its conjugating element would centralize the two noncommuting basis elements $a$ and $y$, and hence would be trivial; this contradicts $\Phi^n(x)\ne x$. Thus $\tau$ again has infinite order. Parts (1) and (2) now imply
\[
\tau\notin\Stab_{\Out(F_N)}(\nu)
=\Stab_{\Out(F_N)}([\nu]),
\]
which proves (3).
\end{proof}

\begin{defn}
Let $F_N$ be free of finite rank $N\ge 2$, let $A$ be a free basis of $F_N$ and let $\nu\in \CN$ be a filling current.

(1) We call the set $\mathfrak W_{A,\nu}:=\{\phi\in\Out(F_N): ||\phi\nu||_A=J_A(\nu)\}$ the \emph{$A$-minimizing set for $\nu$} and we call the integer $M_A(\nu):=\#\mathfrak W_{A,\nu}\ge 1$  the \emph{minimizing multiplicity} for $\nu$ with respect to $A$.

(2) Also let 
\[
J'_A(\nu)=\min\{||\phi\nu||_A: \phi\in \Out(F_N), \phi\not\in \mathfrak W_{A,\nu}\}=\min \left(D_A(\nu)\setminus\{J_A(\nu)\}\right).
\] 
and let $\lambda_A(\nu)=\frac{J'_A(\nu)}{J_A(\nu)}$, so that $\lambda_A(\nu)>1$. We call $\lambda_A(\nu)$ the \emph{distortion threshold} for $\nu$ with respect to $A$.  

(3) Denote $\Im_{A,\nu}=\mathfrak W_{A,\nu}\nu=\{\phi\nu|\phi\in \mathfrak W_{A,\nu}\}\subseteq \CN$ and call $\Im_{A,\nu}$ the \emph{orbit floor} for $\nu$.

(4) For $\nu'\in \Im_{A,\nu}$ denote $R_{A,\nu,\nu'}=\{\psi\in \Out(F_N): \psi\nu'\in \Im_{A,\nu}\}$. 
\end{defn}

\begin{lem}\label{lem:R}
Let $F_N$ be free of rank $N\ge2$, let $A$ be a free basis, and let $\nu\in\CN$ be filling. Suppose that $\nu'=\phi\nu\in\Im_{A,\nu}$ with $\phi\in\mathfrak W_{A,\nu}$. Then:
\begin{enumerate}
\item
\[
R_{A,\nu,\nu'}=\mathfrak W_{A,\nu}\phi^{-1}
=\{\phi'\phi^{-1}:\phi'\in\mathfrak W_{A,\nu}\}.
\]
In particular, $\#R_{A,\nu,\nu'}=M_A(\nu)$.
\item For every current $\eta$ and every conjugacy class $[u]$,
\[
R_{A,\nu,\nu'}\eta=\mathfrak W_{A,\nu}\phi^{-1}\eta,
\qquad
R_{A,\nu,\nu'}[u]=\mathfrak W_{A,\nu}\phi^{-1}[u].
\]
\item $R_{A,\nu,\nu'}\nu'=\Im_{A,\nu}$.
\end{enumerate}
\end{lem}
\begin{proof}
Put $H=\Stab_{\Out(F_N)}(\nu)$. First observe that the minimizing set is right $H$-invariant:
\[
\mathfrak W_{A,\nu}H=\mathfrak W_{A,\nu},
\]
because $h\nu=\nu$ for $h\in H$. Now $\psi\nu'\in\Im_{A,\nu}$ if and only if there exists $\phi'\in\mathfrak W_{A,\nu}$ such that
\[
\psi\phi\nu=\phi'\nu.
\]
Equivalently, $(\phi')^{-1}\psi\phi\in H$, or
\[
\psi\in\phi'H\phi^{-1}
\subseteq\mathfrak W_{A,\nu}\phi^{-1}.
\]
The converse inclusion is immediate: if $\psi=\theta\phi^{-1}$ with $\theta\in\mathfrak W_{A,\nu}$, then $\psi\nu'=\theta\nu\in\Im_{A,\nu}$. This proves (1), including the cardinality assertion. Part (2) follows by applying the equality of sets of automorphisms, and part (3) follows from
\[
R_{A,\nu,\nu'}\nu'
=\mathfrak W_{A,\nu}\phi^{-1}\phi\nu
=\mathfrak W_{A,\nu}\nu
=\Im_{A,\nu}.\qedhere
\]
\end{proof}

We now obtain the following key technical result of this paper:

\begin{thm}\label{t:key1}
Let $F_N$ be free of rank $N\ge2$, let $A$ be a free basis, and let $\nu\in\CN$ be a filling current. Let
\[
1<\lambda<\lambda_A(\nu),\qquad 0<\epsilon<\lambda-1.
\]
Put $\mathfrak W=\mathfrak W_{A,\nu}$ and $M=M_A(\nu)=\#\mathfrak W$. Then there exists a neighborhood $U=U([\nu],\lambda,\epsilon)$ of $[\nu]$ in $\PCN$ such that, whenever $1\ne w\in F_N$ and $[\eta_w]\in U$, the following hold:
\begin{enumerate}
\item $S=\mathfrak W[w]$ is $(M,\lambda,\epsilon,\mathcal W_N)$-minimizing and $\mathcal M([w])\subseteq S$;
\item every element of $S$ is $M$-minimal;
\item $\Stab_{\Out(F_N)}([w])$ is finite.
\end{enumerate}
\end{thm}
\begin{proof}
Put $\Im=\Im_{A,\nu}=\mathfrak W\nu$. Every $\nu'\in\Im$ has $\|\nu'\|_A=J_A(\nu)$. If $\psi\nu'\notin\Im$, then
\[
\frac{\|\psi\nu'\|_A}{\|\nu'\|_A}\ge\lambda_A(\nu).
\]
Choose numbers $\lambda_1$ and $\delta$ such that
\[
\lambda<\lambda_1<\lambda_A(\nu),\qquad
0<\delta<\epsilon,\qquad
\lambda_1>1+\delta,\qquad
\lambda_1(1-\delta)>1.
\tag{6.1}
\]
For $\nu'\in\Im$, put $R_{\nu'}=R_{A,\nu,\nu'}$. By Lemma~\ref{lem:R}, $\#R_{\nu'}=M$ and $R_{\nu'}\nu'=\Im$.

For fixed $\nu'\in\Im$, the functions
\[
[\eta]\longmapsto\frac{\|\psi\eta\|_A}{\|\eta\|_A}
\]
are continuous on $\PCN$. At $[\nu']$ their value is $1$ for $\psi\in R_{\nu'}$, while for every $\tau\in\mathcal W_N\setminus R_{\nu'}$ their value is at least $\lambda_A(\nu)>\lambda_1$. Since $\Im$, every $R_{\nu'}$, and $\mathcal W_N$ are finite, we may choose neighborhoods $V_{\nu'}$ of $[\nu']$, equivariantly with respect to the finite groupoid of maps $R_{\nu'}$, so that:
\begin{enumerate}
\item[(a)] if $[\eta]\in V_{\nu'}$ and $\psi\in R_{\nu'}$, then $\psi[\eta]\in V_{\psi\nu'}$ and
\[
1-\delta\le\frac{\|\psi\eta\|_A}{\|\eta\|_A}\le1+\delta;
\]
\item[(b)] if $[\eta]\in V_{\nu'}$ and $\tau\in\mathcal W_N\setminus R_{\nu'}$, then
\[
\frac{\|\tau\eta\|_A}{\|\eta\|_A}>\lambda_1.
\]
\end{enumerate}
For completeness, such an equivariant family is obtained by choosing one point $\nu_0\in\Im$, intersecting the pullbacks of the finitely many preliminary neighborhoods to $[\nu_0]$, making the resulting neighborhood invariant under the finite stabilizer of $\nu_0$, and transporting it to the other points of $\Im$.

Finally, shrink a neighborhood $U$ of $[\nu]$ so that
\[
\phi U\subseteq V_{\phi\nu}
\qquad\text{for every }\phi\in\mathfrak W.
\tag{6.2}
\]
Let $[\eta_w]\in U$ and put $S=\mathfrak W[w]$. Clearly $\#S\le M$, and all elements of $S$ belong to the same $\Out(F_N)$-orbit.

Fix $[u]=\phi[w]\in S$ and put $\nu'=\phi\nu$. By (6.2), $[\eta_u]\in V_{\nu'}$, and Lemma~\ref{lem:R}(2) gives $R_{\nu'}[u]=S$. If $[x]\in S$, choose $\psi\in R_{\nu'}$ with $[x]=\psi[u]$. Property (a) gives
\[
1-\delta\le\frac{\|x\|_A}{\|u\|_A}\le1+\delta,
\]
so condition (3) of Definition~\ref{d:MLEW} holds with $\epsilon$.

Let $\tau\in\mathcal W_N$. If $\tau\in R_{\nu'}$, then $\tau[u]\in S$. If $\tau\notin R_{\nu'}$, property (b) gives
\[
\frac{\|\tau(u)\|_A}{\|u\|_A}>\lambda_1>\lambda,
\]
and $\tau[u]\notin S$, since all elements of $S$ have length at most $(1+\delta)\|u\|_A$ while $\lambda_1>1+\delta$. Thus $S$ is $(M,\lambda,\epsilon,\mathcal W_N)$-minimizing. In fact, the same argument shows that it is $(M,\lambda_1,\delta,\mathcal W_N)$-minimizing. Proposition~\ref{prop:M} therefore implies that every element of $S$ is $M$-minimal. Moreover, if $[u]\in S$, then $\Out(F_N)[u]=\Out(F_N)[w]$; hence Lemma~\ref{lem:aux}, applied to $[u]$ with the parameters $\lambda_1,\delta$, gives $\mathcal M([w])=\mathcal M([u])\subseteq S$.

It remains to prove finiteness of the stabilizer. Choose $[u]\in S$ of minimal length, write $[u]=\phi[w]$ with $\phi\in\mathfrak W$, and put $\nu'=\phi\nu$. By (6.2), $[\eta_u]\in V_{\nu'}$. If a Whitehead move $\tau$ lies in $R_{\nu'}$, then $\tau[u]\in S$ and cannot be shorter than $[u]$; if $\tau\notin R_{\nu'}$, property (b) strictly increases length. Thus no Whitehead move decreases $\|u\|_A$, and $[u]$ is $\Out(F_N)$-minimal by Proposition~\ref{p:wpr}.

Let $\psi\in\Stab_{\Out(F_N)}([u])$. By Proposition~\ref{prop:LSST}, write $\psi=\tau_k\cdots\tau_1$ so that every intermediate class
\[
[u_i]=\tau_i\cdots\tau_1([u])
\]
has length $\|u\|_A$. Put $\nu_0=\nu'$ and, inductively, $\nu_i=\tau_i\cdots\tau_1\nu'$ for $1\le i\le k$. Starting with $[\eta_u]\in V_{\nu_0}$, property (b) rules out $\tau_{i+1}\notin R_{\nu_i}$, because every intermediate conjugacy class has length $\|u\|_A$, while property (a) transports the counting current to $V_{\nu_{i+1}}$. Thus, inductively,
\[
\nu_i=\tau_i\cdots\tau_1\nu'\in\Im
\qquad(1\le i\le k).
\]
Hence $\psi\nu'\in\Im$, so $\psi\in R_{\nu'}$. Therefore
\[
\#\Stab_{\Out(F_N)}([u])\le\#R_{\nu'}=M.
\]
The stabilizers of $[u]$ and $[w]$ are conjugate in $\Out(F_N)$, completing the proof.
\end{proof}

\begin{cor}\label{cor:key1}
Let $F_N$ be free of finite rank $N\ge 2$, let $A$ be a free basis of $F_N$ and let $\nu\in \CN$ be a filling current. Let $\lambda$ be such that $1<\lambda<\lambda_A(\nu)$ and let $0<\epsilon<1$ be such that $\lambda_A(\nu) >\lambda>1+\epsilon$.

Let  $\mathfrak W=\mathfrak W_{A,\nu}$ and let $M=M_A(\nu)=\#\mathfrak W$.

Then there exists a neighborhood $U_1=U_1([\nu], \lambda,\epsilon)$ of $[\nu]$ in $\PCN$ such that for every $1\ne w\in F_N$ with $[\eta_w]\in U_1$  the following hold:

\begin{enumerate}
\item For every $\phi\in \mathfrak W$ the element $\phi[w]$ is $M$-minimal.
\item The set $S=\mathfrak W[w]\subseteq \mathcal C_N$ is $(M,\lambda,\epsilon)$-minimizing.
\item We have $\mathcal M([w])\subseteq S$, and hence $\#\mathcal M([w])\le M$.
\item The stabilizer $\Stab_{\Out(F_N)}([w])$ is finite.
\end{enumerate}
\end{cor}
\begin{proof}
First choose $\lambda'$ such that $\lambda_A(\nu)>\lambda'>\lambda>1$. Then choose $\epsilon'$ such that $0<\epsilon'<\min\{\epsilon,\lambda'-1\}$ and $\lambda'(1-\epsilon')>\lambda$.  By Theorem~\ref{t:key1}, there exists a neighborhood $U=U([\nu], \lambda',\epsilon')$ of $[\nu]$ in $\PCN$ such that for every $1\ne w\in F_N$ with $[\eta_w]\in U$ the set $S=\mathfrak W[w]$ is $(M,\lambda',\epsilon',\mathcal W_N)$-minimizing. Therefore, by Proposition~\ref{prop:cont}, the set $S$ is $(M,\lambda,\epsilon)$-minimizing. Theorem~\ref{t:key1} also gives $\mathcal M([w])\subseteq S$ and finiteness of $\Stab_{\Out(F_N)}([w])$.
Then for any $\phi\in \mathfrak W$ the element $\phi([w])$ is $(M,\lambda',\epsilon',\mathcal W_N)$-minimal. Therefore Proposition~\ref{prop:M} implies that $\phi[w]$ is $M$-minimal.
Thus $U_1:=U$ satisfies the requirements of the corollary.
\end{proof}

\begin{rem}\label{rem:v}
Suppose that $U\subseteq\PCN$ is a neighborhood of $[\nu]$ provided by Theorem~\ref{t:key1}. The coordinate description of the weak-* topology gives a finite set $\mathbb V\subseteq F(A)-\{1\}$ and a number $\epsilon_0>0$ such that
\[
U_0:=\left\{[\eta]\in\PCN:
\left|
\frac{\langle v,\eta\rangle_A}{\|\eta\|_A}
-\frac{\langle v,\nu\rangle_A}{\|\nu\|_A}
\right|<\epsilon_0
\text{ for every }v\in\mathbb V
\right\}
\]
is a neighborhood of $[\nu]$ contained in $U$. Consequently, if $1\ne w\in F_N$ satisfies
\[
\left|
\frac{\langle v,w\rangle_A}{\|w\|_A}
-\frac{\langle v,\nu\rangle_A}{\|\nu\|_A}
\right|<\epsilon_0
\qquad(v\in\mathbb V),
\tag{$\spadesuit$}
\]
then $[\eta_w]\in U_0\subseteq U$, and the conclusions of Theorem~\ref{t:key1} apply to $w$.
\end{rem}

\begin{defn}\label{d:adapt}
Let $\mathcal W=W_1, W_2, \dots, W_n, \dots$ be a sequence of $F_N$-valued random variables.

\begin{enumerate}
\item We say that $\mathcal W$ is \emph{tame} if for some (equivalently, any) free basis $A$ of $F_N$ there exists $C>0$ such that we always have $|W_n|_A\le Cn$ where $n\ge 1$.
\item Let $0\ne \nu\in \CN$.  We say that the sequence $\mathcal W$ is \emph{$\nu$-adapted} if, for almost every trajectory $w_1,w_2,\dots,w_n,\dots$ of $\mathcal W$, one has
\[
\lim_{n\to\infty} [\eta_{w_n}]=[\nu]
\]
in $\PCN$.
\end{enumerate}
\end{defn}

In Definition~\ref{d:adapt} above, a random trajectory of $\mathcal W$ is implicitly required to satisfy $w_n\ne 1$ for all sufficiently large $n$ (which is needed in order for $\eta_{w_n}$ to be defined), but we do not require $||w_n||_A\to\infty$ as $n\to\infty$.
In particular, if $\nu=\eta_w$ for some $1\ne w\in F_N$, and the random process $\mathcal W$ always outputs $W_n=w$ for all $n\ge 1$, then $\mathcal W$ is $\nu$-adapted.

The following statement is key for our paper:

\begin{prop}\label{prop:key}
Let $F_N$, $A$, $\nu$, $\lambda$, $\epsilon$, $\mathfrak W$, and $M$ be as in Corollary~\ref{cor:key1}, and let $\mathcal W=W_1,W_2,\dots$ be a $\nu$-adapted sequence of $F_N$-valued random variables. Then:
\begin{enumerate}
\item for almost every trajectory, all sufficiently large $n$ satisfy all conclusions of Corollary~\ref{cor:key1} with $w=W_n$;
\item the probability that all those conclusions hold at time $n$ tends to $1$ as $n\to\infty$.
\end{enumerate}
\end{prop}
\begin{proof}
Let $U_1$ be the neighborhood supplied by Corollary~\ref{cor:key1}. Since $\mathcal W$ is $\nu$-adapted, for almost every trajectory one has $[\eta_{W_n}]\in U_1$ for all sufficiently large $n$. Almost-sure convergence also implies convergence in probability, so
\[
\Pr\bigl([\eta_{W_n}]\in U_1\bigr)\longrightarrow1.
\]
Corollary~\ref{cor:key1} gives the result.
\end{proof}

\begin{thm}\label{t:A}
Let $F_N=F(A)$ be a free group of finite rank $N\ge 2$ with a free basis $A$.

Let $\mathcal W=W_1, W_2,\dots$ be a sequence of $F(A)$-valued random variables. Let $0\ne\nu\in \CN$ be a filling geodesic current such that $\mathcal W$ is adapted to $\nu$.

Then there exist $M\ge 1$, a number $\lambda$ with $1<\lambda<2$, and a subset $\mathfrak W\subseteq \Out(F_N)$ with $\#\mathfrak W\le M$ such that for every $0<\epsilon<1$ with $\lambda>1+\epsilon$ the following hold:

\begin{itemize}
\item[(a)] For almost every trajectory $\xi=(w_1,w_2,\dots, w_n, \dots )$ of $\mathcal W$ there exists $n_0=n_0(\xi)\ge 1$ such that the following hold for all $n\ge n_0$:

\begin{enumerate}
\item  The set $S_n=\mathfrak W [w_n]$ is $(M,\lambda,\epsilon)$-minimizing.
\item For every $\phi\in \mathfrak W$ the conjugacy class $\phi[w_n]\in S_n$ is  $(M,\lambda,\epsilon)$-minimal and $M$-minimal.
\item We have $\mathcal M([w_n])\subseteq \mathfrak W[w_n]$, and in particular, $\#\mathcal M([w_n])\le M$.
\item The stabilizer $\Stab_{\Out(F_N)}([w_n])$ is finite.
\end{enumerate}

\item[(b)] The probability of each of the following events tends to $1$ as $n\to\infty$:
\begin{enumerate}
\item  The set $S_n=\mathfrak W [W_n]$ is $(M,\lambda,\epsilon)$-minimizing.
\item For every $\phi\in \mathfrak W$ the conjugacy class $\phi[W_n]$ is  $(M,\lambda,\epsilon)$-minimal and $M$-minimal.
\item We have $\mathcal M([W_n])\subseteq \mathfrak W[W_n]$, and $\#\mathcal M([W_n])\le M$.
\item The stabilizer $\Stab_{\Out(F_N)}([W_n])$ is finite.
\end{enumerate}
\end{itemize}

\end{thm}
\begin{proof}
Put $\mathfrak W=\mathfrak W_{A,\nu}$ and $M=M_A(\nu)$. Choose $\lambda$ with
\[
1<\lambda<\min\{\lambda_A(\nu),2\}.
\]
If $0<\epsilon<\lambda-1$, then $\epsilon<1$ and $\lambda>1+\epsilon$. Proposition~\ref{prop:key}, together with Corollary~\ref{cor:key1}, therefore gives every assertion in (a) and (b).
\end{proof}

\begin{thm}\label{t:A'}
Let $F_N=F(A)$, $\nu$, $M$, $\lambda$, $\mathfrak W$, and $\mathcal W=W_1,W_2,\dots$ be as in Theorem~\ref{t:A}, and assume that $\mathcal W$ is tame. Then there exists $K_0\ge1$ such that the following hold.

\begin{itemize}
\item[(a)] For almost every pair of independently sampled trajectories
\[
\xi=(w_1,w_2,\dots),\qquad
\xi'=(w_1',w_2',\dots),
\]
there exist $n_0,m_0\ge1$ such that:
\begin{enumerate}
\item for every $n\ge n_0$, the $\mathfrak W$-speed-up of Whitehead minimization on $w_n$ terminates in time at most $K_0n$ and produces an element of $\mathcal M([w_n])$;
\item for every $n\ge n_0$ and every $u\in F_N$, the $\mathfrak W$-speed-up of Whitehead's algorithm decides in time at most
\[
K_0\max\{n,|u|_A^2\}
\]
whether $\Aut(F_N)w_n=\Aut(F_N)u$;
\item for every $n\ge n_0$ and $m\ge m_0$, the $\mathfrak W$-speed-up of Whitehead's algorithm decides in time at most
\[
K_0\max\{n,m\}
\]
whether $\Aut(F_N)w_n=\Aut(F_N)w_m'$.
\end{enumerate}

\item[(b)] The following convergence-in-probability assertions hold.
\begin{enumerate}
\item The probability that the $\mathfrak W$-speed-up of Whitehead minimization on $W_n$ terminates in time at most $K_0n$ and produces an element of $\mathcal M([W_n])$ tends to $1$.
\item The probability that $W_n$ has the following property tends to $1$: for every $u\in F_N$, the $\mathfrak W$-speed-up of Whitehead's algorithm decides in time at most $K_0\max\{n,|u|_A^2\}$ whether $\Aut(F_N)W_n=\Aut(F_N)u$.
\item Let $\mathcal W'=W_1',W_2',\dots$ be an independent copy of $\mathcal W$, and let $n_i,m_i\ge1$ satisfy $\min\{n_i,m_i\}\to\infty$. Then the probability that the $\mathfrak W$-speed-up decides in time at most $K_0\max\{n_i,m_i\}$ whether
\[
\Aut(F_N)W_{n_i}=\Aut(F_N)W_{m_i}'
\]
tends to $1$ as $i\to\infty$.
\end{enumerate}
\end{itemize}
\end{thm}
\begin{proof}
Choose $C\ge1$ so that $|W_n|_A\le Cn$ for every $n$. For each $\phi\in\mathfrak W$, choose a representative $\Phi_\phi\in\Aut(F_N)$. Since $\mathfrak W$ is finite, there exists $L\ge1$ such that
\[
|\Phi_\phi(g)|_A\le L|g|_A
\qquad(\phi\in\mathfrak W,\ g\in F_N).
\tag{6.3}
\]
All constants below depend only on $N$, $M$, $\mathfrak W$, the chosen representatives, and the tameness constant $C$.

Fix $\epsilon_0$ with $0<\epsilon_0<\min\{1,\lambda-1\}$. Call an index $n$ \emph{good} for a trajectory if the four pointwise conclusions in Theorem~\ref{t:A}(a), with $\epsilon=\epsilon_0$, hold for $w_n$. In particular, at a good index every class in $\mathfrak W[w_n]$ is $M$-minimal and
\[
\#\mathcal M([w_n])\le M.
\tag{6.4}
\]
By Theorem~\ref{t:A}, for almost every trajectory all sufficiently large indices are good, and the probability that $n$ is good tends to $1$.

Suppose that $n$ is good. For every $\phi\in\mathfrak W$, the auxiliary branch beginning at $\Phi_\phi(w_n)$ starts from an $M$-minimal class and from a word of length at most $L|w_n|_A$. Theorem~\ref{t:WHM}(a), together with the finite round-robin overhead, therefore makes the speeded-up minimization terminate in $O(|w_n|_A)$ time. Its output is an orbit-minimal class $[v_n]\in\mathcal M([w_n])$ with $\|v_n\|_A=O(|w_n|_A)$. Since $|w_n|_A\le Cn$, this gives the $O(n)$ estimate underlying (a)(1) and (b)(1).

Now let $u\ne1$ be arbitrary. The speeded-up minimization of $u$ includes the ordinary branch and hence terminates in $O(|u|_A^2)$ time, producing some $[v]\in\mathcal M([u])$. If $\|v_n\|_A\ne\|v\|_A$, automorphic inequivalence is already decided. If the two minimal lengths agree, (6.4) says that the stabilization component
\[
V\mathcal T_{\|v_n\|_A}[v_n]=\mathcal M([w_n])
\]
has at most $M$ vertices. Its construction and the final membership test take $O(\|v_n\|_A)=O(|w_n|_A)$ time. Thus the total running time is
\[
O\bigl(\max\{|w_n|_A,|u|_A^2\}\bigr),
\]
and tameness gives the stated bound $O(\max\{n,|u|_A^2\})$. The case $u=1$ is immediate. This proves the deterministic estimate underlying (a)(2) and (b)(2).

If $n$ and $m$ are good for two trajectories, the two speeded-up minimizations take $O(|w_n|_A)$ and $O(|w_m'|_A)$ time, respectively. The first stabilization component has at most $M$ vertices by (6.4), so the remaining work takes $O(\max\{|w_n|_A,|w_m'|_A\})$. Tameness gives $O(\max\{n,m\})$, proving the deterministic estimate underlying (a)(3).

Choose $K_0$ larger than the finitely many implied constants. The almost-sure eventual goodness in Theorem~\ref{t:A}(a) proves part (a). For (b)(1),(2), apply the same deterministic bounds on the good event and use Theorem~\ref{t:A}(b). For (b)(3), if $G_n$ and $G_m'$ denote the good events for the two copies, then
\[
\Pr(G_{n_i}\cap G_{m_i}')
\ge1-\Pr(G_{n_i}^c)-\Pr((G_{m_i}')^c)\longrightarrow1.
\]
The deterministic two-input estimate on this intersection completes the proof.
\end{proof}

\section{\texorpdfstring{Group random walks as a source of $M$-minimality}{Group random walks as a source of M-minimality}}

\begin{conv}[Terminology regarding random processes]
Let $B$ be a set with the discrete topology (such as a discrete group, the set of vertices of a graph, words in a finite alphabet, etc).
For any infinite sequence of $B$-valued random variables $\mathcal W=W_1, W_2, \dots, W_n, \dots$ we assume that the sample space $\Omega=B^\omega$ (as usual given the product topology for the discrete topologies on the factors $B$)  is a probability space equipped with a Borel probability measure $\Pr$. We will usually suppress the explicit mention of this probability measure $\Pr$.
Thus a trajectory of $\mathcal W$ is a sequence $\zeta=(w_1, w_2,\dots, w_n, \dots)\in\Omega$, where all $w_i\in B$. We say that some property $\mathcal E$ \emph{holds  for almost every trajectory} of $\mathcal W$ if \[\Pr\bigl(\{\zeta\in\Omega:\zeta\text{ satisfies }\mathcal E\}\bigr)=1.\]
\end{conv}

\begin{conv}
For a discrete probability measure $\mu:G\to [0,1]$ on a group $G$, we denote by $\sm$ the subsemigroup of $G$ generated by the support $\supp(\mu)$ of $\mu$. Note that we have $\sm=\cup_{n=1}^\infty \supp(\mu^{(n)})$ where $\mu^{(n)}$ is the $n$-fold convolution of $\mu$.
Thus for $g\in G$ we have $g\in \sm$ if and only if there exist $n\ge 1$ and $g_1,\dots, g_n\in G$ such that $g=g_1\dots g_n$ and $\mu(g_i)>0$ for $i=1,\dots, n$.
\end{conv}

\begin{defn}[Group random walk]
Let $G$ be a group and let $\mu:G\to [0,1]$ be a  discrete probability measure on $G$. Let $X_1, X_2, \dots, X_n,\dots$ be a sequence of $G$-valued i.i.d. random variables, where each $X_i$ has distribution $\mu$. Put $W_n=X_1\dots X_n\in G$, where $n=1,2\dots$. The random process
\[
\mathcal W=W_1,W_2,\dots, W_n,\dots
\]
is called the \emph{random walk on $G$} defined by $\mu$.
\end{defn}

Recall that if $G$ is a group acting on a set $X$, and $\mu$ is a discrete probability measure on $G$, then a measure $\lambda$ on $X$ is called \emph{$\mu$-stationary} if $\lambda=\sum_{g\in G} \mu(g) g\lambda$.

If $G$ is a non-elementary word-hyperbolic group, a discrete probability measure $\mu$ on $G$ is called \emph{non-elementary} if $\sm$ contains some two independent loxodromic elements of $G$ (which, for a word-hyperbolic $G$ means some two elements $g_1,g_2\in G$ of infinite order such that $\langle g_1\rangle\cap \langle g_2\rangle=\{1\}$).

We need the following well-known fact (see, e.g. \cite[Theorem~1.1]{MT18} for the most general version of this statement for random walks on groups acting on Gromov-hyperbolic spaces; see \cite[Theorem 7.6]{Kaim00} specifically for the case of a word-hyperbolic $G$):
\begin{prop}\label{p:ne}
Let $G$ be a non-elementary word-hyperbolic group and let $\mu$ be a non-elementary discrete probability measure on $G$. Let $\mathcal W=W_1,W_2,\dots, W_n,\dots$ be the random walk on $G$ defined by $\mu$. Then:
\begin{enumerate}
\item For almost every trajectory $w_1,w_2,\dots $ of $\mathcal W$ there exists $x\in\partial G$ such that $\lim_{n\to\infty} w_n=x$ in $G\cup \partial G$.
\item Putting, for $S\subseteq \partial G$, $\lambda(S)$ to be the probability that a trajectory of $\mathcal W$ converges to a point of $S$, defines a $\mu$-stationary Borel probability measure $\lambda$ on $\partial G$.
\end{enumerate}
\end{prop}
This measure $\lambda$ is called the \emph{exit measure} or the \emph{hitting measure} for $\mathcal W$.

Recall also that if $G$ is a word-hyperbolic group and $H\le G$ is a non-elementary subgroup, then $\partial G$ contains a unique nonempty minimal closed $H$-invariant subset $\Lambda(H)\subseteq \partial G$ called the \emph{limit set of $H$} (see \cite{KS96,KB02} for details).

We need the following fact which appears to be folklore, although it does not seem to appear in the literature. We include a proof, explained to us by Vadim Kaimanovich, for completeness.

\begin{prop}\label{p:supp}
Let $G$ be a non-elementary word-hyperbolic group, let $\mu$ be a non-elementary discrete probability measure on $G$, and let $\lambda$ be the exit measure on $\partial G$ for the random walk on $G$ defined by $\mu$.

Suppose $H\le G$ is a non-elementary subgroup such that $H\subseteq \sm$. Then $\Lambda(H)\subseteq \supp(\lambda)$.

In particular if $\Lambda(H)=\partial G$ then $\supp(\lambda)=\partial G$.
\end{prop}

\begin{proof}
Let $\lambda$ be the exit measure on $\partial G$ for the random walk determined by $\mu$.
For any $k\ge 1$, the measure $\lambda$ is also an exit measure for the random walk based on $\mu^{(k)}$, and therefore $\lambda$ is  $\mu^{(k)}$-stationary.
Thus for every $n\ge 1$ we have $\lambda=\sum_{g\in G} \mu^{(n)}(g)\cdot  g\lambda$.  Hence $\lambda$ dominates $g\lambda$ whenever $n\ge 1$ and $\mu^{(n)}(g)>0$, that is, whenever $g\in\sm$. Since $H\subseteq \sm$, it follows that $\lambda$ dominates $h\lambda$ for every $h\in H$. Since $H$ is a subgroup of $G$, this implies that for all $h\in H$ the measures $\lambda$ and $h\lambda$ are in the same measure class. Hence for every $h\in H$ $\supp(\lambda)=h\supp(\lambda)$. Thus $\supp(\lambda)$ is a nonempty closed $H$-invariant subset of $\partial G$, and therefore $\Lambda(H)\subseteq \supp(\lambda)$, as claimed.
\end{proof}
Note that if $\sm$ contains a subgroup $H$ of $G$ such that $H$ has finite index in $G$, or such that $H$ is an infinite normal subgroup of $G$, then $\Lambda(H)=\partial G$ (see \cite{KS96}) and therefore we get $\supp(\lambda)=\partial G$ in the conclusion of Proposition~\ref{p:supp}.

\begin{thm}\label{t:rwa}
Let $F_N=F(A)$ be a free group of rank $N\ge2$, and let $\mu:F_N\to[0,1]$ be a finitely supported probability measure such that $\langle\supp(\mu)\rangle_+=F_N$. Let $\mathcal W=W_1,W_2,\dots$ be the random walk defined by $\mu$. Then $\mathcal W$ is tame and is adapted to a filling current $0\ne\nu\in\CN$.
\end{thm}
\begin{proof}
Let $T_A$ be the Cayley tree of $F_N$ with respect to $A$. Since $\mu$ has finite support,
\[
C:=\max\{|g|_A:\mu(g)>0\}<\infty,
\]
and $|W_n|_A\le Cn$ for every $n$. Thus $\mathcal W$ is tame.

The semigroup hypothesis implies that $\mu$ is non-elementary. Let
\[
\check\mu(g)=\mu(g^{-1}).
\]
Since the inverse of a product reverses the order of its factors,
\[
\langle\supp(\check\mu)\rangle_+
=\{g^{-1}:g\in\langle\supp(\mu)\rangle_+\}
=F_N.
\]
Thus $\check\mu$ is non-elementary as well. Let $\lambda$ and $\check\lambda$ be the hitting measures of the $\mu$- and $\check\mu$-walks. Applying Proposition~\ref{p:supp} with $H=F_N$, whose limit set is $\partial F_N$, gives
\[
\supp(\lambda)=\supp(\check\lambda)=\partial F_N.
\tag{7.1}
\]

The action of $F_N$ on $T_A$ is proper and cocompact, hence convex cocompact. The measure $\mu$ has finite second moment, and Gekhtman's Axiom~1.4 holds for convex-cocompact actions and finitely supported measures. Therefore \cite[Theorem~1.5]{Ge17} applies. For almost every trajectory, $W_n$ is loxodromic for all sufficiently large $n$, and the normalized arclength measures $D_{W_n}$ on the corresponding oriented closed geodesics in $T^1(T_A/F_N)$ converge weakly to the harmonic invariant probability measure $m$.

Let $\mathcal I_A$ be the current--flow correspondence in (4.1), and let $\bar\nu\in\Curr^+(F_N)$ be the nonzero oriented current determined by
\[
\mathcal I_A(\bar\nu)=m.
\]
The $F_N$-invariant lift of $m$ to $T^1(T_A)\cong\partial^2F_N\times\R$ is, under the suspension correspondence, $\bar\nu\times dt$. Axiom~1.4 says that this lift is in the measure class of
\[
\check\lambda\times\lambda\times dt.
\]
Uniqueness of the disintegration along flow lines therefore gives
\[
\bar\nu\asymp
(\check\lambda\times\lambda)|_{\partial^2F_N},
\tag{7.2}
\]
where $\asymp$ denotes equality of measure classes.

For every sufficiently large $n$, (4.1) gives
\[
\mathcal I_A\left(\frac{\eta_{W_n}^+}{\|W_n\|_A}\right)=D_{W_n}.
\]
Since $\mathcal I_A$ is a homeomorphism, Gekhtman's convergence implies that, for almost every trajectory,
\[
\frac{1}{\|W_n\|_A}\eta_{W_n}^+
\longrightarrow\bar\nu
\]
in the weak-* topology on oriented currents. The continuous symmetrization map then gives
\[
\frac{1}{\|W_n\|_A}\eta_{W_n}
=\operatorname{Sym}\left(\frac{1}{\|W_n\|_A}\eta_{W_n}^+\right)
\longrightarrow
\nu:=\bar\nu+\varpi_*\bar\nu.
\tag{7.3}
\]
Consequently $[\eta_{W_n}]\to[\nu]$, and $\mathcal W$ is adapted to the usual flip-invariant current $\nu$.

By (7.2), the current $\nu$ is in the measure class of
\[
(\check\lambda\times\lambda)+(\lambda\times\check\lambda)
\quad\text{on }\partial^2F_N.
\]
In view of (7.1), every nonempty open subset of $\partial^2F_N$ has positive measure for this symmetrized product measure. Hence $\nu$ has full support. Proposition~\ref{p:fs} now implies that $\nu$ is filling.
\end{proof}

We can now conclude that algebraic and algorithmic conclusions of Theorem~\ref{t:A} and Theorem~\ref{t:A'}  apply in the case of $\mu$-random walks on $F_N$, where $\mu$ has finite support with $\sm=F_N$:

\begin{cor}
Let $F_N=F(A)$ be free of rank $N\ge 2$, with a free basis $A$. Let $\mu:F_N\to [0,1]$ be a finitely supported discrete probability measure such that $\sm=F_N$.

Let $\mathcal W=W_1,\dots, W_n, \dots $ be the  random walk on $F_N$ defined by $\mu$.

Then there exist $M\ge 1$, $0<\epsilon<1$ and $\lambda>1+\epsilon$ and a subset $\mathfrak W\subseteq \Out(F_N)$ with $\#\mathfrak W\le M$ such that the conclusions of Theorem~\ref{t:A} and Theorem~\ref{t:A'} hold for $\mathcal W$ with these choices of $M,\lambda,\epsilon,\mathfrak W$.
\end{cor}

\section{Finite-state Markov chains and the frequency measures}

We recall some basic notions and facts regarding finite-state Markov chains here and refer the reader to \cite{DZ98,Ga13,Ki15,Ku18} for proofs and additional details.

\subsection{Finite-state Markov chains.}  Recall that a \emph{finite-state Markov chain, or FSMC} $\mathcal X$ is defined by a finite nonempty set $S$ of \emph{states} and by a family of \emph{transition probabilities} $p_\mathcal X(s,s')\ge 0$, where $s,s'\in S$ such that for every $s\in S$ $\sum_{s'\in S} p_\mathcal X(s,s')=1$.  Then for every integer $n\ge 1$ we also get the \emph{$n$-step transition probabilities} $p_\mathcal X^{(n)}(s,s')$ where $p_\mathcal X^{(1)}(s,s')=p_\mathcal X(s,s')$ and where for $n\ge 2$ and $s,s'\in S$ we have
\[
p_\mathcal X^{(n)}(s,s')=\sum_{s''\in S} p_\mathcal X^{(n-1)}(s,s'')p_\mathcal X(s'',s').
\]
The \emph{sample space} associated with $\mathcal X$ is the product space $S^\N=\{\xi=(s_1,s_2,s_3,\dots, s_n, \dots)| s_i\in S \text{ for } i\ge 1\}$. The set $S$ is given the discrete topology and $S^\N$ is given the corresponding product topology, which makes $S^\N$ a compact metrizable totally disconnected topological space. For $i\ge 1$ we denote by $X_i:S^\N\to S$ the function picking out the $i$-th coordinate of an element of $S^\N$.
The \emph{transition matrix} $M=M(\mathcal X)$ is an $S\times S$ matrix where for $s,s'\in S$ the entry $M(s,s')$ of $M$ is defined as $M(s,s')=p_\mathcal X(s,s')$. Thus $M(\mathcal X)$ is a nonnegative matrix, where the sum of the entries in each row is equal to $1$.  Also, for all $n\ge 1$ and $s,s'\in S$ we have $p_\mathcal X^{(n)}(s,s')=(M^n)(s,s')$.  A FSMC $\mathcal X$ as above is called \emph{irreducible} if for all $s,s'\in S$ there exists $n\ge 1$ such that $p_\mathcal X^{(n)}(s,s')>0$. Thus $\mathcal X$ is irreducible if and only if the nonnegative matrix $M(\mathcal X)$ is irreducible in the sense of Perron-Frobenius theory.

For an FSMC $\mathcal X$, given a \emph{initial probability distribution} $\mu$ on $S$,  we obtain the corresponding \emph{Markov process} $\mathcal X_{\mu}=X_1,\dots, X_n,\dots $ where each $X_i$ is an $S$-valued random variable with probability distribution $\mu_i$ on $S$, where $\mu_1=\mu$ and where for $i\ge 2$ and $s'\in S$ we have $\mu_i(s')=\sum_{s\in S} \mu_{i-1}(s)p_\mathcal X(s,s')$. An initial distribution $\mu$ on $S$ is called \emph{stationary} for $\mathcal X$ if $\mu_i=\mu$ for all $i\ge 1$ (equivalently, if $\mu_2=\mu$).
It is well-known, by the basic result of Perron-Frobenius theory, that if $\mathcal X$ is an irreducible finite-state Markov chain with state set $S$, then there is a unique stationary probability distribution $\mu$ on $S$ for $\mathcal X$, and that it satisfies $\mu(s)>0$ for all $s\in S$. In this case the row vector $(\mu(s))_{s\in S}$ is the unique strictly positive left eigenvector of $M(\mathcal X)$ having $||.||_1$-norm $1$ and eigenvalue $1$. Moreover, $1$ is the Perron--Frobenius eigenvalue of $M(\mathcal X)$; in particular, it is simple and equals the spectral radius.

For an FSMC $\mathcal X$ with state set $S$, a word $w=s_1\dots s_n\in S^n$ of length $n\ge 2$ is called \emph{feasible}  if $p_\mathcal X(s_1,s_2)\dots p_\mathcal X(s_{n-1},s_n)>0$. Also, we consider all words $w=s\in S$ of length $n=1$ to be \emph{feasible}. (Hence every nonempty subword of a feasible word is also feasible). An element $\xi=(s_1,s_2,\dots )\in S^\N$ is \emph{feasible} for $\mathcal X$ if for every $n\ge 1$ the word $s_1\dots s_n$ is feasible. Denote by $(S^\N)_+$ the set of all feasible $\xi\in S^\N$. Also, for every $n\ge 1$ denote by $(S^n)_+$ the set of all feasible $s_1\dots s_n\in S^n$.

For a word $w=s_1\dots s_n\in S^n$ (where $n\ge 2$) put
\[
p_\mathcal X(w):=p_\mathcal X(s_1,s_2)\dots p_\mathcal X(s_{n-1},s_n)
\]

Any initial probability distribution $\mu$ on $S$ defines a Borel probability $\mu_\infty$ via the standard convolution formulas. Namely, if $n\ge 1, s_1,\dots s_n\in S$ then
\[
\mu_\infty \left(Cyl(s_1\dots s_n)\right):=\mu(s_1)p_\mathcal X(s_1,s_2)\dots p_\mathcal X(s_{n-1},s_n)=\mu(s_1)p_\mathcal X(s_1\dots s_n)
\]
where $Cyl(s_1\dots s_n)=\{\xi\in S^\N| X_i(\xi)=s_i \text{ for } i=1,\dots, n\}$.

If $\mu$ is strictly positive on $S$, then the support $\supp(\mu_\infty)$ of $\mu_\infty$ is equal to $(S^\N)_+$. In particular, that is the case if $\mathcal X$ is an irreducible FSMC and $\mu$ is the unique stationary probability distribution on $S$.

\begin{defn}[Occurrences and frequencies]\label{d:occ}
Let $\mathcal X$ be an irreducible finite-state Markov chain with state set $S$.

(1) For a word $w=s_1\dots s_n\in S^n$ (where $n\ge 1$) and an element $s\in S$ we denote by $\langle s,w\rangle$ the number of those $i\in\{1,\dots, n\}$ such that $s_i=s$. We call $\langle s,w\rangle$ the \emph{number of occurrences of $s$ in $w$}. We also put $\theta_s(w)=\frac{\langle s,w\rangle}{|w|}$, where $|w|=n$ is the length of $w$. We call $\theta_s(w)$ the \emph{frequency} of $s$ in $w$.

(2) The above notions can be extended from $s$ to arbitrary nonempty words $v\in S^\ast$ as follows. Let $v=y_1\dots y_m\in S^m$ where $y_j\in S$ for $j=1,\dots, m$. Also denote by $w^\infty$ the semi-infinite word $w^\infty=wwww\dots $. For an arbitrary integer $i\ge 1$ we still denote by $s_i\in S$ the $i$-th letter of $w^\infty$.  Now define $\langle v,w\rangle$ to be the number of $i\in\{1,\dots, n\}$ such that in $w^\infty$ we have $s_i=y_1, s_{i+1}=y_2,\dots, s_{i+m-1}=y_m$. We call $\langle v,w\rangle$ the \emph{number of occurrences of  $v$ in $w$}, and we call $\theta_v(w)=\frac{\langle v,w\rangle}{|w|}$ the \emph{frequency of $v$ in $w$}.
\end{defn}

We record the following immediate corollary of the above definition (which holds since we defined the numbers of occurrences in $w$ cyclically).

\begin{lem}\label{lem:occ}
Let $w\in S^n$ where $n\ge 1$. Then the following hold:
\begin{enumerate}
\item We have $n=|w|=\sum_{s\in S} \langle s,w\rangle$ and $1=\sum_{s\in S} \theta_s(w)$.
\item For every $m\ge 1$ we have $n=|w|=\sum_{v\in S^m} \langle v,w\rangle$ and $1=\sum_{v\in S^m} \theta_v(w)$.
\item For every $m\ge 1$ and every $v\in S^m$ we have
\[
\langle v,w\rangle=\sum_{s\in S} \langle vs,w\rangle=\sum_{s'\in S} \langle s'v,w\rangle.
\]
and
\[
\theta_v(w)=\sum_{s\in S} \theta_{vs}(w)=\sum_{s'\in S} \theta_{s'v}(w).
\]
\end{enumerate}
\end{lem}

For a finite-state Markov chain $\mathcal X$ with state set $S$ and an element $\xi=(s_1,s_2,s_3,\dots, s_n, \dots)$ of $S^\N$, we denote $w_n=s_1\dots s_n\in S^n$, where $n\ge 1$.

The strong law of large numbers for finite-state Markov chains implies:

\begin{prop}\label{prop:fr}
Let  $\mathcal X$ be an irreducible  finite-state Markov chain with state set $S$ and let $\mu_0$ be the unique stationary probability distribution on $S$. Let $\mu$ be an arbitrary initial distribution on $S$ defining the corresponding Markov process  $\mathcal X_{\mu}=X_1,\dots, X_n,\dots $. Then the following hold:

\begin{enumerate}
\item For every $s\in S$  and for $\mu_\infty$-a.e. trajectory $\xi=(s_1,s_2,s_3,\dots, s_n, \dots)\in S^\N$ of $\mathcal X_{\mu}$, we have
\[
\lim_{n\to\infty} \theta_s(w_n)=\mu_0(s).
\]
\item For every $0<\epsilon\le 1$ and every $s\in S$
\[
\lim_{n\to\infty} \Pr_{\mu_\infty}( |\theta_s(w_n)-\mu_0(s)|<\epsilon)=1
\]
and the convergence in this limit is exponentially fast as $n\to\infty$.
\end{enumerate}

\end{prop}

\subsection{Iterated Markov Chains}

Let $\mathcal X$ be a finite-state Markov chain with state set $S$. Let $k\ge 1$ be an integer. Consider a finite-state Markov chain $\mathcal X[k]$ with the state set $(S^k)_+$ and with transition probabilities defined as follows.  Suppose $s_1\dots s_k\in  (S^k)_+$ and $s\in S$ are such that $p_\mathcal X(s_k,s)>0$ (so that $s_1\dots s_k s\in S^{k+1}$ is feasible for $\mathcal X$, and $s_2\dots s_ks\in (S^k)_+$). Then put $p_{\mathcal X[k]}(s_1\dots s_k, s_2\dots s_ks)=p_\mathcal X(s_k,s)$. Set all other transition probabilities in $\mathcal X[k]$ to be $0$.  Note that we have $\mathcal X[1]=\mathcal X$.

It is not hard to see that if $\mathcal X$ as above is irreducible then for every $k\ge 1$ the FSMC $\mathcal X[k]$ is also irreducible. Moreover, in this case there is a natural canonical homeomorphism between the set of infinite feasible trajectories $(S^\N)_+$ of $\mathcal X$ and the set $\left(((S^k)_+)^\N\right)_+$ of infinite feasible trajectories for $\mathcal X[k]$. Under this homeomorphism a sequence $\xi=(s_1,\dots, s_n\dots)\in (S^\N)_+$ goes to $(v_1,v_2\dots, v_n,\dots)\in  \left(((S^k)_+)^\N\right)_+$ where $v_i=s_is_{i+1}\dots s_{i+k-1}$.
Moreover, if $\mu_0$ is the unique stationary distribution for $\mathcal X$ on $S$ then
\[
\mu_0[k](s_1\dots s_k)=\mu_0(s_1)p_\mathcal X(s_1,s_2)\dots p_\mathcal X(s_{k-1},s_k), \tag{*}
\] where $s_1\dots s_k\in (S^k)_+$, is the unique stationary probability distribution for $\mathcal X[k]$.  Using these facts and the application of Proposition~\ref{prop:fr}, standard results about Markov chains imply the following statement; see \cite[Proposition~3.13]{CMah} for a more detailed version of this statement, with explicit speed of convergence estimates:

\begin{prop}\label{prop:fr1}
Let $\mathcal X$ be an irreducible finite-state Markov chain with state set $S$, let $\mu_0$ be its stationary distribution, and let $\mu$ be any initial distribution. For $k\ge1$, extend
\[
\mu_0[k](s_1\dots s_k)
=\mu_0(s_1)p_\mathcal X(s_1,s_2)\cdots p_\mathcal X(s_{k-1},s_k)
\]
by zero to nonfeasible words. Then, for every $v\in S^k$:
\begin{enumerate}
\item for $\mu_\infty$-almost every trajectory,
\[
\theta_v(w_n)\longrightarrow\mu_0[k](v);
\]
\item for every $\epsilon>0$,
\[
\Pr\bigl(|\theta_v(w_n)-\mu_0[k](v)|<\epsilon\bigr)\longrightarrow1,
\]
and the convergence is exponentially fast.
\end{enumerate}
\end{prop}
\begin{proof}
Apply Proposition~\ref{prop:fr} to the irreducible block chain $\mathcal X[k]$. The ordinary number of length-$k$ occurrences counted by that chain and the cyclic occurrence number in Definition~\ref{d:occ} differ by at most $k-1$. After division by $n$, this discrepancy tends to zero, and the same bounded discrepancy does not affect the exponential convergence in probability.
\end{proof}

\begin{cor}\label{cor:freq}
In the setting of Proposition~\ref{prop:fr1}, for every $m\ge1$ one has
\[
\sum_{v\in S^m}\mu_0[m](v)=1,
\]
and, for every $v\in S^m$,
\[
\mu_0[m](v)
=\sum_{s\in S}\mu_0[m+1](vs)
=\sum_{s\in S}\mu_0[m+1](sv).
\]
\end{cor}
\begin{proof}
Choose a trajectory in the countable intersection of the full-measure sets supplied by Proposition~\ref{prop:fr1} for all finite words, and pass to the limit in Lemma~\ref{lem:occ}(2),(3).
\end{proof}

\subsection{Quasi-inversions}

An irreducible finite-state Markov chain is called a \emph{deterministic cycle} if every state has a unique successor and the corresponding transition probability is $1$. Otherwise it is called \emph{non-deterministic}.

\begin{prop}\label{p:tight}
Let $\mathcal X$ be an irreducible non-deterministic finite-state Markov chain with state set $S$, $\#S\ge2$. Then there exist an integer $q\ge1$ and a number $0<\sigma<1$ with the following property. Let $S'\subseteq S$, let $\iota:S'\to S$ be injective, and let $\mu$ be any initial distribution. Put $r_n=\lfloor\sqrt n\rfloor$. Then, for all sufficiently large $n$,
\[
\Pr\left(
X_1\cdots X_{r_n}
=\bigl(\iota(X_{n-r_n+1})\cdots\iota(X_n)\bigr)^R
\right)
\le \sigma^{\lfloor r_n/q\rfloor},
\tag{8.1}
\]
where the event is declared empty if one of the last $r_n$ states does not belong to $S'$. In particular, the probabilities in (8.1) are summable in $n$.
\end{prop}
\begin{proof}
Because $\mathcal X$ is finite, irreducible, and not a deterministic cycle, there is a uniform bound on the length of a directed path all of whose transition probabilities are $1$: otherwise a repeated state would produce a closed deterministic communicating class, which by irreducibility would be the entire chain. Hence there exists $q\ge1$ such that every feasible block of $q$ consecutive transitions has probability at most some $\sigma<1$.

Condition on $X_1,\dots,X_{n-r_n}$. If the event in (8.1) is possible, it prescribes the next $r_n$ states. Its conditional probability is a product of transition probabilities and is at most $\sigma^{\lfloor r_n/q\rfloor}$. Averaging over the conditioning gives (8.1). The resulting sequence is summable because $r_n\asymp\sqrt n$.
\end{proof}

\begin{cor}\label{c:qi}
In the setting of Proposition~\ref{p:tight}, for almost every trajectory the equality in (8.1) fails for all sufficiently large $n$.
\end{cor}
\begin{proof}
This is the first Borel--Cantelli lemma.
\end{proof}

\begin{rem}
If $\mathcal X$ is a deterministic cycle, every trajectory is periodic. In the graph-based setting below, one period projects to a fixed closed reduced and cyclically reduced path. The proof of Theorem~\ref{t:cla} gives a periodic-word argument showing that cyclic reduction removes only a uniformly bounded number of edges from the closed prefixes.
\end{rem}

\section{Graph-based non-backtracking random walks}

\begin{conv}\label{conv:g}
In this section we will assume that $F_N=F(A)$ is a free group of finite rank $N\ge 2$, that $\Gamma$ is a finite connected oriented graph with all vertices of degree $\ge 3$ and with the first Betti number $b(\Gamma)=N$, and that $\alpha: F_N\to \pi_1(\Gamma,x_0)$ is a fixed isomorphism, where $x_0\in V\Gamma$ is some base vertex. We equip $\Gamma$ and $T_0=\widetilde \Gamma$ with simplicial metrics, where every edge has length $1$.
\end{conv}
Note that for $\Gamma$ as above we always have $\#E\Gamma\le 6N$.

\begin{defn}
Under the above convention, a FSMC $\mathcal X$  with state set $S$ is \emph{$\Gamma$-based} if the following hold:

\begin{enumerate}
\item We have $S\subseteq E\Gamma$, with $\#S\ge 2$.
\item Whenever $e,e'\in S$ are such that $p_\mathcal X(e,e')>0$ then $t(e)=o(e')$ in $\Gamma$ and $e'\ne e^{-1}$.
\end{enumerate}
\end{defn}

Thus for a $\Gamma$-based FSMC $\mathcal X$ as above, the space of feasible trajectories $(S^\N)_+$ can be thought of as a subset of the set $\Omega(\Gamma)$ of all reduced semi-infinite edge-paths $\gamma=e_1,e_2,\dots $ in $\Gamma$. Similarly, $(S^n)_+$ can be thought of as a subset of the set $\Omega_n(\Gamma)$ of all reduced length $n$ edge-paths $e_1,e_2,\dots, e_n$ in $\Gamma$.

\begin{prop}\label{p:n}
Let $\mathcal X$ be an irreducible $\Gamma$-based FSMC with state set $S\subseteq E\Gamma$.  Let $\mu_0$ be the unique stationary probability distribution on $S$.  For every $k\ge 1$ we extend $\mu_0[k]$ to $\Omega_k(\Gamma)$ by setting $\mu_0[k](v)=0$ for every $v\in  \Omega_k(\Gamma)-(S^k)_+$.

  There exists a unique geodesic current $\nu$ on $F_N$ with the following properties:
\begin{enumerate}
\item For every $k\ge 1$ and every $v\in \Omega_k(\Gamma)$ we have $\langle v,\nu\rangle_\Gamma=\mu_0[k](v)+\mu_0[k](v^{-1})$.
\item We have $\langle T_0, \nu\rangle=1$.
\end{enumerate}
\end{prop}
\begin{proof}
We use the formulas in part (1) of the proposition to define a system of  weights $\nu$ on $\cup_{n\ge 1} \Omega_n(\Gamma)$.
Note that these weights are already symmetrized since the defining equations for the weights in (1) give the same answers for $v$ and $v^{-1}$.
Now Corollary~\ref{cor:freq} implies that these $\nu$ weights satisfy the switch conditions. Therefore they do define a geodesic current $\nu\in \CN$.

Also, part (1) of Corollary~\ref{cor:freq} implies that $\sum_{e\in E\Gamma} \mu_0[1](e)=1=\sum_{e\in E\Gamma} \mu_0[1](e^{-1})$. For the unit-edge tree $T_0$, the intersection form is one half of the sum of the weights of all oriented edges. Therefore
\[
\langle T_0, \nu\rangle=\frac{1}{2} \sum_{e\in E\Gamma} (\mu_0[1](e)+\mu_0[1](e^{-1}))=\frac{1}{2} \sum_{e\in E\Gamma}\langle e, \nu\rangle_\Gamma=1.
\]
This proves the normalization in part (2) and, in particular, $\nu\ne0$. Uniqueness follows from the coordinate characterization of currents above.
\end{proof}

\begin{defn}[Characteristic current]
Let $\mathcal X$ be an irreducible $\Gamma$-based FSMC with state set $S\subseteq E\Gamma$.  Let $0\ne\nu\in\CN$ be the geodesic current constructed in Proposition~\ref{p:n} above.
We call $\nu$ the \emph{characteristic current} of $\mathcal X$ and denote it by $\nu_\mathcal X$.
\end{defn}

\begin{defn}[$\mathcal X$-directed random walk on $\Gamma$]
Let $\mathcal X$ be an irreducible $\Gamma$-based FSMC with state set $S\subseteq E\Gamma$.  Let $\mu$ be any initial probability distribution on $S$
defining the corresponding Markov process  $\mathcal X_{\mu}=X_1,\dots, X_n,\dots $. For every $n=1,2,\dots $ put $W_n=X_1\dots X_n$ so that $W_n$ takes values in $S^n$.
The random process $\mathcal W_{\mu}=W_1,\dots, W_n,\dots $ is called the \emph{$\mathcal X$-directed non-backtracking random walk} on $\Gamma$ corresponding to $\mu$.
\end{defn}
Note that for $\mathcal W_{\mu}$ and any $n\ge 1$ the only feasible values of $W_n$ are contained in $S^n\cap \Omega_n(\Gamma)$.

Since in general $\Gamma$ may have more than one vertex, a reduced edge-path in $\Gamma$ (such as, for example, the length-$n$ path given by $W_n$ in the above setting) is not necessarily closed and thus may not define a conjugacy class in $\pi_1(\Gamma,x_0)$. To get around this issue, we modify $\mathcal W_{\mu}$ slightly, in two different ways to output closed paths in $\Gamma$.

\begin{defn}[Closing path system]
Let $\Gamma$ be as in Convention~\ref{conv:g}.  A \emph{closing path system} for $\Gamma$ is a family $\mathcal B=(\beta_{e,e'})_{e,e'\in E\Gamma}$ of reduced edge-paths in $\Gamma$ such that for every $e,e'\in E\Gamma$ $e\beta_{e,e'}e'$ is a reduced edge-path in $\Gamma$.

For a nondegenerate reduced edge-path $\gamma$ in $\Gamma$ define the \emph{$\mathcal B$-closing} $\widehat \gamma$ of $\gamma$ as $\widehat\gamma=\gamma\beta_{e,e'}$ where $e$ is the last edge of $\gamma$ and $e'$ is the first edge of $\gamma$. Note also that for any nondegenerate reduced edge-path $\gamma$ in $\Gamma$ the $\mathcal B$-closing  $\widehat\gamma$ is a reduced and cyclically reduced closed edge-path in $\Gamma$.

\end{defn}
For a closing path system $\mathcal B$ as above, if $e,e'\in E\Gamma$ then $t(e)=o(\beta_{e,e'})$ and $o(e')=t(\beta_{e,e'})$. Such a system always exists. Indeed, the directed non-backtracking edge graph, whose vertices are the elements of $E\Gamma$ and whose arrows are the reduced pairs $ef$, is strongly connected because $\Gamma$ is connected and every vertex has degree at least $3$. A shortest directed path from $e$ to $e'$ supplies the required intermediate path $\beta_{e,e'}$; when $e=e'$, choose a shortest positive directed return path. Such a path can be chosen without repeated directed-edge vertices except for its two endpoints, and hence
\[
|\beta_{e,e'}|\le |E\Gamma|\le6N.
\]

\begin{defn}[$\mathcal B$-closing of a non-backtracking walk on $\Gamma$]
Let $\mathcal X$ be an irreducible $\Gamma$-based FSMC with state set $S\subseteq E\Gamma$.  Let $\mathcal B=(\beta_{e,e'})_{e,e'\in E\Gamma}$ be a closing path system for $\Gamma$.
Let $\mu$ be any initial probability distribution on $S$ and let $\mathcal W_{\mu}=W_1,\dots, W_n,\dots $ be the $\mathcal X$-directed non-backtracking random walk on $\Gamma$ corresponding to $\mu$.
Define the random process $\widehat {\mathcal W_{\mu}}=\widehat W_1,\dots, \widehat W_n,\dots $, where $\widehat W_n$ is the $\mathcal B$-closing of $W_n$.
We call $\widehat {\mathcal W_{\mu}}$ the \emph{$\mathcal B$-closing} of $\mathcal W_{\mu}$.
\end{defn}

An advantage of using $\widehat {\mathcal W_{\mu}}$ is that it always outputs reduced and cyclically reduced closed paths $\widehat W_n$ of length $n\le |\widehat W_n|\le n+C$, where $C=\max_{e,e'}|\beta_{e,e'}|$.  In many natural examples, however, $W_n$ is already closed with positive probability. Therefore we offer a variation of the $\widehat {\mathcal W_{\mu}}$ approach which takes this fact into account.

For a reduced nondegenerate closed edge-path $\gamma$ in $\Gamma$ denote by $\operatorname{cyc}(\gamma)$ the subpath of $\gamma$ obtained from $\gamma$ by a maximal cyclic reduction. Thus $\operatorname{cyc}(\gamma)$ is a nondegenerate closed reduced and cyclically reduced edge-path in $\Gamma$.

\begin{nota}
Let $\mathcal B$ be a closing path system for $\Gamma$. For a nondegenerate reduced edge-path $\gamma$ in $\Gamma$ let $\breve\gamma:=\operatorname{cyc}(\gamma)$ if $\gamma$ is a closed path, and let $\breve\gamma:=\widehat\gamma$ otherwise.  Thus in both cases $\breve\gamma$ is a closed reduced and cyclically reduced edge-path in $\Gamma$ (but it may now have length $<n$).  We call $\breve\gamma$  the \emph{modified $\mathcal B$-closing} of $\gamma$.
\end{nota}

\begin{defn}
Let $\mathcal X$ be an irreducible $\Gamma$-based FSMC with state set $S\subseteq E\Gamma$.  Let $\mathcal B=(\beta_{e,e'})_{e,e'\in E\Gamma}$ be a closing path system for $\Gamma$.
Let $\mu$ be any initial probability distribution on $S$. Let $\mathcal W_{\mu}=W_1,\dots, W_n,\dots $ be the $\mathcal X$-directed non-backtracking random walk on $\Gamma$ corresponding to $\mu$.

Define the random process $\breve {\mathcal W_{\mu}}=\breve W_1,\dots, \breve W_n,\dots $, where $\breve W_n$ is the modified $\mathcal B$-closing of $W_n$.
We call $\breve {\mathcal W_{\mu}}$ the \emph{modified $\mathcal B$-closing} of $\mathcal W_{\mu}$.

\end{defn}

\begin{conv}\label{conv:closedpaths}
For each vertex $x\in V\Gamma$, choose a reduced path $\gamma_x$ from $x_0$ to $x$, with $\gamma_{x_0}$ trivial. If $\rho$ is a closed path based at $x$, define
\[
g_\rho:=\alpha^{-1}\bigl([\gamma_x\rho\gamma_x^{-1}]\bigr)\in F_N,
\]
and choose its freely reduced representative over $A^{\pm1}$. We use the same symbol $\rho$ for this representative when no confusion can arise. A different choice of the paths $\gamma_x$ changes $g_\rho$ by conjugation, so it does not change the counting current $\eta_\rho$ or any statement about automorphic equivalence of conjugacy classes.

Because the marking and the graph are fixed, there exist constants $K_\Gamma\ge1$ and $K_\Gamma'\ge0$ such that every closed edge-path $\rho$ satisfies
\[
|g_\rho|_A\le K_\Gamma|\rho|+K_\Gamma'.
\tag{9.1}
\]
Thus the closed-path processes defined above are canonically regarded, up to harmless conjugacy choices, as $F_N$-valued random processes.
\end{conv}

\begin{thm}\label{t:cl}
Let $\mathcal X$ be an irreducible $\Gamma$-based FSMC, let $\mu$ be any initial distribution, let $\mathcal B$ be a closing path system, and let $\nu_\mathcal X$ be the characteristic current. Then the $\mathcal B$-closing $\widehat{\mathcal W_\mu}$ is tame and adapted to $\nu_\mathcal X$.
\end{thm}
\begin{proof}
Put $C=\max_{e,e'}|\beta_{e,e'}|$. In graph length,
\[
n\le|\widehat W_n|\le n+C.
\]
By (9.1), $|\widehat W_n|_A\le K_\Gamma(n+C)+K_\Gamma'$, so the process is tame.

Fix a trajectory $\xi=e_1e_2\dots$ to which Proposition~\ref{prop:fr1} applies for every finite edge-path, and put $w_n=e_1\dots e_n$. Here $\langle v,w_n\rangle$ denotes the one-direction cyclic occurrence number of Definition~\ref{d:occ}; for a closed cyclically reduced path $z$ one has
\[
\langle v,\eta_z\rangle_\Gamma
=\langle v,z\rangle+\langle v^{-1},z\rangle.
\]
For a reduced path $v$ of length $k$,
\[
\left|\langle v,\widehat w_n\rangle-\langle v,w_n\rangle\right|\le C+2k,
\]
because only occurrences meeting the added closing path or one of its two endpoints can change. The same estimate holds for $v^{-1}$. Hence
\begin{align*}
\lim_{n\to\infty}\frac{\langle v,\eta_{\widehat w_n}\rangle_\Gamma}{n}
&=\lim_{n\to\infty}
\frac{\langle v,w_n\rangle+\langle v^{-1},w_n\rangle}{n}\\
&=\mu_0[k](v)+\mu_0[k](v^{-1})
=\langle v,\nu_\mathcal X\rangle_\Gamma.
\end{align*}
The coordinate characterization of currents gives
\[
\frac1n\eta_{\widehat w_n}\longrightarrow\nu_\mathcal X,
\]
and therefore $[\eta_{\widehat w_n}]\to[\nu_\mathcal X]$.
\end{proof}

\begin{thm}\label{t:cla}
Let $\mathcal X$ be an irreducible $\Gamma$-based FSMC, let $\mu$ be any initial distribution, let $\mathcal B$ be a closing path system, and let $\nu_\mathcal X$ be the characteristic current. Then the modified closing $\breve{\mathcal W_\mu}$ is tame and adapted to $\nu_\mathcal X$.
\end{thm}
\begin{proof}
Put $C=\max_{e,e'}|\beta_{e,e'}|$. Since $|\breve W_n|\le n+C$ in graph length, (9.1) gives
\[
|\breve W_n|_A\le K_\Gamma(n+C)+K_\Gamma',
\]
so the process is tame. Fix a trajectory to which Proposition~\ref{prop:fr1} applies for every finite path, and write $w_n=e_1\cdots e_n$. Whenever $w_n$ is closed, let $c_n$ be the number of pairs of boundary edges removed in passing from $w_n$ to $\operatorname{cyc}(w_n)$.

Suppose first that $\mathcal X$ is non-deterministic. If $c_n\ge\lfloor\sqrt n\rfloor$, then the first $\lfloor\sqrt n\rfloor$ edges are the inverse reverse of the last $\lfloor\sqrt n\rfloor$ edges. Apply Corollary~\ref{c:qi} with $\iota(e)=e^{-1}$ on the subset of states whose inverse also lies in $S$. Almost surely,
\[
c_n<\sqrt n
\]
for every sufficiently large closed $w_n$. Thus
\[
|\operatorname{cyc}(w_n)|\ge n-2\sqrt n.
\]

Suppose now that $\mathcal X$ is a deterministic cycle. After cyclically shifting the period to begin with the first state of the trajectory, write the successive prefixes as $P^qR$, where $P$ is a nontrivial closed reduced and cyclically reduced period and $R$ ranges over the finitely many prefixes of $P$. Fix such an $R$ for which $P^qR$ is closed. We claim that the amount of cyclic cancellation in $P^qR$ is bounded independently of $q$.

Indeed, suppose that the cancellation were unbounded. If $\ell_q$ pairs of boundary edges cancel in $P^qR$, then the initial segment of length $\ell_q$ of the periodic ray $P^\infty$ agrees with the inverse reverse of the terminal segment of length $\ell_q$ of $P^qR$. Removing the fixed terminal path $R$ loses at most $|R|$ letters, and the remaining terminal segment lies in one of the finitely many phases of the periodic word $P^\infty$. Passing to a subsequence for which this phase is fixed and $\ell_q\to\infty$, we obtain agreement of arbitrarily long prefixes of $P^\infty$ with prefixes of a fixed shift of $(P^{-1})^\infty$. Hence these one-sided periodic words agree, and their primitive periods are cyclic conjugates. Thus the primitive cyclic root $Q$ of $P$ is cyclically conjugate to $Q^{-1}$.

The nontrivial element represented by $Q$ would therefore be conjugate to its inverse. Any such conjugator normalizes the maximal cyclic subgroup containing $Q$. In a free group the normalizer of a nontrivial maximal cyclic subgroup is the subgroup itself, so the conjugator centralizes $Q$. It would follow that $Q=Q^{-1}$, and hence $Q^2=1$, contradicting torsion-freeness. Thus cancellation is bounded for each $R$, and the finiteness of the possible remainders makes the bound uniform. Consequently $c_n=O(1)$ along the closed times in the deterministic case. Thus, in either case, $c_n=o(n)$ along every sequence of closed times.

Let $v$ be a reduced path of length $k$. Removing $2c_n=o(n)$ boundary edges changes the cyclic number of occurrences of $v$ by at most $2c_n+2k=o(n)$. Therefore, along every sequence of closed times $n_i$,
\begin{align*}
\lim_{i\to\infty}
\frac{\langle v,\eta_{\breve w_{n_i}}\rangle_\Gamma}{n_i}
&=\lim_{i\to\infty}
\frac{\langle v,w_{n_i}\rangle+\langle v^{-1},w_{n_i}\rangle}{n_i}\\
&=\mu_0[k](v)+\mu_0[k](v^{-1})
=\langle v,\nu_\mathcal X\rangle_\Gamma.
\end{align*}
At nonclosed times $\breve w_n=\widehat w_n$, and the coordinate convergence proved in Theorem~\ref{t:cl} applies. The closed and nonclosed subsequences therefore have the same coordinate limit, so
\[
\frac1n\eta_{\breve w_n}\longrightarrow\nu_\mathcal X
\quad\text{and hence}\quad
[\eta_{\breve w_n}]\longrightarrow[\nu_\mathcal X].
\]
\end{proof}

In summary, we get:

\begin{cor}\label{c:sum}
Let $\mathcal X$ be an irreducible $\Gamma$-based FSMC with state set $S\subseteq E\Gamma$. Let $\mu$ be any initial probability distribution on $S$. Let $\mathcal B=(\beta_{e,e'})_{e,e'\in E\Gamma}$ be a closing path system for $\Gamma$. Let $\nu_\mathcal X\in\CN$ be the characteristic current for $\mathcal X$. Suppose that $\nu_\mathcal X$ is filling in $F_N$.

Then $\widehat{\mathcal W_\mu}$ and $\breve{\mathcal W_\mu}$ are adapted to the characteristic current $\nu_\mathcal X$. Therefore Theorem~\ref{t:A} and Theorem~\ref{t:A'}  apply to $\widehat{\mathcal W_\mu}$ and $\breve{\mathcal W_\mu}$.
\end{cor}

We next explain several situations where one can guarantee that the current $\nu_\mathcal X\in\CN$ is filling.

\begin{prop}\label{p:XF}
Let $\mathcal X$ be an irreducible $\Gamma$-based FSMC with state set $S\subseteq E\Gamma$. Let $\nu_\mathcal X\in \CN$ be the characteristic current for $\mathcal X$.

\begin{enumerate}
\item Suppose that $\mathcal X$ has the property that $S=E\Gamma$ and that for every $e,e'\in E\Gamma$ such that $ee'$ is a reduced edge-path in $\Gamma$ we have $p_\mathcal X(e,e')>0$. Then the current $\nu_\mathcal X\in\CN$ is filling.
\item Suppose that $\Gamma=R_A$, the $N$-rose corresponding to a free basis $A=\{a_1,\dots, a_N\}$ of $F_N$ (so that we can identify $E(R_A)=A^{\pm 1}$).  Suppose that $\mathcal X$ is such $A\subseteq S$ and that for all $1\le i,j\le N$ we have $p_\mathcal X(a_i,a_j)>0$.  Then the current $\nu_\mathcal X\in\CN$ is filling.
\item Suppose there exists a nondegenerate reduced cyclically reduced closed edge-path $w$ in $\Gamma$ such that $w$ represents a filling element of $F_N$ and that for every $n\ge 2$ we have $p_\mathcal X(w^n)>0$.    Then the current $\nu_\mathcal X\in\CN$ is filling.
\item Suppose there exists a free basis $A=\{a_1,\dots,a_N\}$ such that the following hold. For $i=1,\dots, N$ let $w_i$ be a closed reduced and cyclically reduced edge-path in $\Gamma$ representing the conjugacy class of $a_i$ in $F_N$. For $1\le i<j\le N$ let $w_{i,j}$ be a closed reduced and cyclically reduced edge-path in $\Gamma$ representing the conjugacy class of $a_i a_j$ in $F_N$. Suppose that we have $p_\mathcal X(w_i^2)>0$ for $i=1,\dots, N$ and that we have $p_\mathcal X(w_{i,j}^2)>0$ for all $1\le i<j\le N$. Then  the current $\nu_\mathcal X\in\CN$ is filling.
\end{enumerate}
\end{prop}

\begin{proof}

Let $\mu_0$ be the unique stationary probability distribution on $S$ for $\mathcal X$.

(1) The assumption on $\mathcal X$ implies that for every reduced edge-path $v$ in $\Gamma$ of length $k\ge 1$ we have $\mu_0[k](v)>0$, and therefore, by definition of $\nu_\mathcal X$, we also have $\langle v,\nu_\mathcal X\rangle_\Gamma >0$. Thus $\nu_\mathcal X\in\CN$ has full support and therefore $\nu_\mathcal X$ is filling in $F_N$.

(2) The assumptions on $\mathcal X$ (with $\Gamma=R_A$) imply that, for every $n\ge1$,
\[
\mu_0[n](a_i^n)>0 \quad(1\le i\le N),
\qquad
\mu_0[2n]((a_i a_j)^n)>0 \quad(1\le i<j\le N).
\]
Therefore
\[
\langle a_i^n,\nu_\mathcal X\rangle_A>0 \quad(1\le i\le N),
\qquad
\langle(a_i a_j)^n,\nu_\mathcal X\rangle_A>0 \quad(1\le i<j\le N).
\]
Proposition~\ref{p:ai} now implies that $\nu_\mathcal X$ is filling.

(3) Again, similarly to (1) and (2) we see that for every $n\ge 1$ $\langle w^n,\nu_\mathcal X\rangle_\Gamma>0$. Therefore by Corollary~\ref{c:z} the current $\nu_\mathcal X\in\CN$ is filling.

(4) Recall that for a reduced edge-path $v$ in $\Gamma$ of length $k\ge 2$ and starting with $e_1\in E\Gamma$ we have $\mu_0[k](v)=\mu_0(e_1)p_\mathcal X(v)$.  Thus $\mu_0[k](v)>0$ if and only if $e_1\in S$ and the transition probabilities $p_\mathcal X(e',e'')$ are $>0$ for all length-2 subpaths $e'e''$ of $v$. Note also that if for the second edge $e_2$ of $v$ we have $p_\mathcal X(e_1,e_2)>0$ then $e_1,e_2\in S$.

Let $\ell_i=|w_i|$ and $\ell_{i,j}=|w_{i,j}|$. The assumptions in part (4) imply that for every $n\ge1$,
\[
\mu_0[n\ell_i](w_i^n)>0,
\qquad
\mu_0[n\ell_{i,j}](w_{i,j}^n)>0.
\] Therefore, by definition of $\nu_\mathcal X$, we have $\langle w_i^n, \nu_\mathcal X\rangle_\Gamma>0$, and, also, for all $1\le i<j\le N$ and all $n\ge 1$ we have $\langle w_{i,j}^n, \nu_\mathcal X\rangle_\Gamma>0$. Therefore, by Proposition~\ref{p:ai}, the current $\nu_\mathcal X\in\CN$ is filling.

\end{proof}

\begin{ex}
Let $A=\{a_1,\dots, a_N\}$ be a free basis of $F_N=F(A)$ and let $\Gamma=R_A$ be the corresponding $N$-rose.

(1) Consider an $R_A$-based  FSMC $\mathcal X$ with state set $S=A^{\pm 1}$ and transition probabilities $p_\mathcal X(a_i^{\epsilon}, a_j^{\delta})=\frac{1}{2N-1}$ if $a_i^{\epsilon}\ne a_j^{-\delta}$ and $p_\mathcal X(a_i^{\epsilon}, a_i^{-\epsilon})=0$, where $\epsilon,\delta=\pm 1$.  Then $\mathcal X$ is irreducible and non-deterministic. The stationary distribution $\mu_0$ is the uniform probability distribution on $A^{\pm 1}$. Then $\mathcal X$, with an initial distribution $\mu$ on $A^{\pm 1}$,  defines the standard non-backtracking simple random walk $\mathcal W_{\mu}=W_1,W_2\dots$ on $F_N=F(A)$. In this case the characteristic current $\nu_\mathcal X$ is the uniform current $\nu_A$ corresponding to $A$.  The current $\nu_\mathcal X=\nu_A$ has full support and therefore is filling. Since $R_A$ has one vertex, $W_n$ is always closed, but in general $\breve W_n=\operatorname{cyc}(W_n)$. Theorem~\ref{t:cla} shows that cyclic reduction changes only a sublinear number of letters almost surely. Using a closing path system $\mathcal B$ produces cyclically reduced words $\widehat W_n=W_n\beta$, where $\beta\in\mathcal B$ is an appropriate closing path. Since $\eta_{W_n}=\eta_{\operatorname{cyc}(W_n)}$, Theorem~\ref{t:cla} also shows directly that $W_n$ is adapted to $\nu_A$; this fact is explained in more detail in \cite{Ka07} and exploited in the context of Whitehead's algorithm there. In this case $\nu_A$ already has the ``strict minimality'' properties similar to those of strictly minimal elements of $F_N$. Again see \cite{Ka07} for details.

(2) Let $\Gamma$ be a simplicial chart on $F_N$. Consider a $\Gamma$-based FSMC $\mathcal X$ with state set $S=E\Gamma$ and transition probabilities satisfying $p_\mathcal X(e,e')>0$ if and only if $ee'$ is a reduced length-2 edge-path in $\Gamma$. Then $\mathcal X$ is irreducible and non-deterministic. The characteristic current $\nu_\mathcal X$  has full support, and therefore is filling.

(3) Let $\mathcal X$ be an $R_A$-based FSMC with state set $S=A$ and transition probabilities satisfying $p_\mathcal X(a_i,a_j)>0$ for all $1\le i,j\le N$. Then $\mathcal X$ is irreducible and non-deterministic. The characteristic current $\nu_\mathcal X$ has the property that  for $1\ne v\in F(A)$ we have $\langle v,\nu_\mathcal X\rangle_A>0$ if and only if $v$ or $v^{-1}$ is a positive word over $A$. The current $\nu_\mathcal X$ is filling in $F_N$ by Proposition~\ref{p:ai}. We again have $W_n=\breve W_n$ in this case,  and moreover, $W_n$ is already cyclically reduced because it is a positive word.

(4) Assume here that $N\ge3$, and let $\Gamma$ be a ``fan of lollipops''. Thus $\Gamma$ has a central vertex $x_0$, oriented edges $e_1,\dots,e_N$ from $x_0$ to distinct vertices $y_i$, and a loop $f_i$ at each $y_i$. Choose the marking so that the based loop $e_i f_i e_i^{-1}$ represents $a_i$. Let $\mathcal X$ have state set
\[
S=E\Gamma-\{f_1^{-1},\dots,f_N^{-1}\},
\]
and assume that $p_\mathcal X(e,e')>0$ whenever $e,e'\in S$ and $ee'$ is reduced. This chain is irreducible and non-deterministic.

For $1\le i\le N$, put $w_i=f_i$, a closed path based at $y_i$. For $1\le i<j\le N$, put
\[
w_{i,j}=f_i e_i^{-1}e_j f_j e_j^{-1}e_i,
\]
a closed reduced and cyclically reduced path based at $y_i$. The path $w_i$ represents the conjugacy class of $a_i$, while conjugating $w_{i,j}$ to the base vertex $x_0$ and freely reducing gives
\[
(e_i f_i e_i^{-1})(e_j f_j e_j^{-1}),
\]
so $w_{i,j}$ represents the conjugacy class of $a_i a_j$. Every transition occurring in $w_i^2$ and $w_{i,j}^2$ has positive probability by construction. Proposition~\ref{p:XF}(4) therefore shows that $\nu_\mathcal X$ is filling.

In (1), (2), (3) and (4) above, the processes $\widehat{\mathcal W_\mu}$ and $\breve{\mathcal W_\mu}$ (where $\mu$ is any initial distribution on the state set $S$ of $\mathcal X$) are adapted to the characteristic current $\nu_\mathcal X$ of the defining irreducible FSMC, and $\nu_\mathcal X$ is filling in $F_N$. Therefore Theorem~\ref{t:A} and Theorem~\ref{t:A'}  apply to $\widehat{\mathcal W_\mu}$ and $\breve{\mathcal W_\mu}$ in these cases.
\end{ex}

\vspace{1cm}

\end{document}